\documentclass[11pt]{amsart}
\usepackage{float}
\usepackage[usenames]{color}
\usepackage{graphicx}
\usepackage{amssymb}
\usepackage{amscd}
\theoremstyle{plain}
\newtheorem{thm}{Theorem}[section]
\newtheorem{lemma}[thm]{Lemma}

\newtheorem{prop}[thm]{Proposition}

\theoremstyle{definition}

\newtheorem{rmk}[thm]{Remark}
\newtheorem{example}[thm]{Example}

\def\dim{\mathop{\hbox {dim}}\nolimits}

\newcommand{\fra}{\mathfrak{a}}

\newcommand{\frg}{\mathfrak{g}}
\newcommand{\frh}{\mathfrak{h}}

\newcommand{\frk}{\mathfrak{k}}
\newcommand{\frl}{\mathfrak{l}}

\newcommand{\frp}{\mathfrak{p}}
\newcommand{\frq}{\mathfrak{q}}

\newcommand{\frs}{\mathfrak{s}}
\newcommand{\frt}{\mathfrak{t}}
\newcommand{\fru}{\mathfrak{u}}

\newcommand{\bbC}{\mathbb{C}}
\newcommand{\bbH}{\mathbb{H}}

\newcommand{\bbR}{\mathbb{R}}

\newcommand{\bbZ}{\mathbb{Z}}

\newcommand{\caL}{\mathcal{L}}

\newcommand{\caR}{\mathcal{R}}

\newcommand{\be}{\begin {equation}}
\newcommand{\ee}{\end {equation}}
\newcommand{\bp}{\begin {proof}}
\newcommand{\ep}{\end {proof}}

\begin{document}

\title[A finiteness result for Dirac cohomology]
{Unitary representations with Dirac cohomology:\\ finiteness in the real case}

\author{Chao-Ping Dong}
\address[Dong]{Mathematics and Science College, Shanghai Normal University, Shanghai 200234,
P.~R.~China}
\email{chaopindong@163.com}
\thanks{Dong is supported by NSFC grant 11571097 and the China Scholarship Council.}

\abstract{Let $G$ be a complex connected  simple algebraic group with a fixed real form
$\sigma$. Let $G(\bbR)=G^\sigma$ be the corresponding group of real points. This paper reports a finiteness  theorem for the classification of irreducible unitary Harish-Chandra modules of $G(\bbR)$ (up to equivalence) having non-vanishing Dirac cohomology. Moreover, we study the distribution of the spin norm along Vogan pencils for certain $G(\bbR)$, with particular attention paid to the unitarily small convex hull introduced by Salamanca-Riba and Vogan.}
\endabstract



\maketitle

\section{Introduction}
%%   Use unstarred sectioning commands to obtain automatic numeration.
%% And don't type dots at the end of a heads.
%%   NEW! The number of accessible sectioning commands was highly restricted.
%% Here is the whole list: \section \subsection \subsubsection. This keeps
%% you from making the structure of text too complex.

In representation theory of Lie groups, the Dirac operator was firstly introduced by  Parthasarathy  to give geometric construction for most of the discrete series representations \cite{P1}, whose algebraic parametrization was achieved by Harish-Chandra \cite{HC1,HC2}. This project was completed by Atiyah and Schmid: they realized all the discrete series in the kernel of the Dirac operator \cite{AS}.

To understand the unitary dual of a real reductive Lie group better, Vogan \cite{Vog97} formulated the notion of Dirac cohomology in 1997, and conjectured that whenever non-zero, Dirac cohomology should reveal the infinitesimal character of the original module. This conjecture was confirmed by Huang and Pand\v zi\'c \cite{HP} in 2002, see Theorem \ref{thm-HP}. Since then,
Dirac cohomology became a new invariant for unitary representations of real reductive Lie groups, and classifying all the irreducible unitary representations (up to equivalence) with non-vanishing Dirac cohomology became an interesting problem which remained open. Among the entire unitary dual, as we shall see from \eqref{Dirac-unitary}, these representations are exactly the extreme ones  in the sense of Parthasarathy's Dirac operator inequality \cite{P1,P2}. Thus understanding them thoroughly should also be important.

An effective way to construct unitary representations is using cohomological induction.  For instance, Salamanca-Riba \cite{Sa} proved that
any irreducible unitary representation with strongly regular infinitesimal character
is cohomologically induced from a one-dimensional representation. Inspired by the work of Huang, Kang and Pand\v zi\'c \cite{HKP}, a formula for Dirac cohomology of cohomologically induced modules was obtained whenever the inducing modules are weakly good \cite{DH}. However, it seems rather hard to get a similar unifying formula  when the weakly good range condition is dropped. This point has perplexed us for quite a long time.
Recently, in the special case of \emph{complex} Lie groups (viewed as real Lie groups), we have proved that beyond the good range, there are at most finitely many irreducible unitary modules with non-zero Dirac cohomology, see Theorem A of \cite{DD}. The first aim of the current paper is to generalize this result to \emph{real} reductive Lie groups.

Now let us be more precise. Let $G$ be a complex connected  simple algebraic group with finite center.
Let $\sigma: G \to G$ be a \emph{real form} of $G$. That is, $\sigma$ is an antiholomorphic Lie group automorphism and $\sigma^2={\rm Id}$. Let $\theta: G\to G$ be the involutive algebraic automorphism of $G$ corresponding to $\sigma$ via Cartan theorem (see Theorem 3.2 of \cite{ALTV}). Put $G(\bbR)=G^{\sigma}$ as the group of real points. Denote by $K:=G^{\theta}$ a maximal compact subgroup of $G$, and put $K(\bbR):=K^{\sigma}$. Denote by $\frg_0$ the Lie algebra of $G(\bbR)$, and let $\frg_0=\frk_0\oplus \frp_0$ be the  Cartan decomposition corresponding to $\theta$ on the Lie algebra Level. Denote by $\frh_{f, 0}=\frt_{f, 0}\oplus \fra_{f, 0}$ the unique $\theta$-stable fundamental Cartan subalgebra of $\frg_0$. That is, $\frt_{f, 0}\subseteq \frk_0$ is maximal abelian.   As usual, we drop the subscripts to stand for the complexified Lie algebras. For example, $\frg=\frg_0\otimes_{\bbR}\bbC$, $\frh_f=\frh_{f, 0}\otimes_{\bbR}\bbC$ and so on. We fix a non-degenerate invariant symmetric bilinear form $B(\cdot, \cdot)$ on $\frg$. Its restrictions to $\frk$, $\frp$, etc., will also be denoted by the same symbol. Then the Dirac cohomology of an irreducible  $(\frg, K(\bbR))$ module $\pi$ is defined as the $\widetilde{K(\bbR)}$-module
\begin{equation}\label{def-Dirac-cohomology}
H_D(\pi)=\text{Ker}\, D/ (\text{Im} \, D \cap \text{Ker} D),
\end{equation}
where $\widetilde{K(\bbR)}$ is the pin double covering group of $K(\bbR)$.
Here the Dirac operator $D$ acts on $\pi\otimes S_G$, and $S_G$ is a spin module of the Clifford algebra $C(\frp)$. We care the most about the case that $\pi$ is unitary. Then $D$ is self-adjoint with respect to a natural inner product on $\pi\otimes S_G$, $\text{Ker} D \cap \text{Im} D=0$, and
\begin{equation}\label{Dirac-unitary}
H_D(\pi)=\text{Ker}\, D=\text{Ker}\, D^2.
\end{equation}
Parthasarathy's Dirac operator inequality now says that $D^2$ has non-negative eigenvalue on any $\widetilde{K(\bbR)}$-type of $\pi\otimes S_G$. Moreover, by Theorem 3.5.2 of \cite{HP2}, it becomes equality on some  $\widetilde{K(\bbR)}$-types of $\pi\otimes S_G$ if and only if $H_D(\pi)$ is non-vanishing.

Let $\widehat{G(\bbR)}^{\mathrm{d}}$ be the set of all equivalence classes of irreducible unitary representations of $G(\bbR)$ with non-zero Dirac cohomology. Our main result is the following.

\medskip
\noindent\textbf{Theorem A.} \,
\emph{For all but finitely many exceptions, any member $\pi$ in  $\widehat{G(\bbR)}^{\mathrm{d}}$
 is cohomologically induced from a member $\pi_{L(\bbR)}$ in $\widehat{L(\bbR)}^{\mathrm{d}}$ which is in the good range. Here $L(\bbR)$ is a proper $\theta$-stable Levi subgroup of $G(\bbR)$.}
\medskip

In the setting of the above theorem, we call the finitely many exceptions the \emph{scattered part} of  $\widehat{G(\bbR)}^{\mathrm{d}}$. For a fixed $G(\bbR)$, our proof of Theorem A will actually give a method to pin down the scattered part of $\widehat{G(\bbR)}^{\mathrm{d}}$. To figure out the other members of $\widehat{G(\bbR)}^{\mathrm{d}}$, Theorem A says that it boils down to look at $\widehat{L(\bbR)}^{\mathrm{d}}$ for the finitely many $\theta$-stable Levis of $G(\bbR)$. Working within the good range, we can do cohomological induction in stages, see Corollary 11.173 of Knapp and Vogan \cite{KV}. Thus, like the special case of complex Lie groups \cite{DD, D17}, we have that $\widehat{L(\bbR)}^{\mathrm{d}}$ is built up with central unitary characters and the scattered part of $\widehat{L(\bbR)}^{\mathrm{d}}_{\rm ss}$ which contains only \emph{finitely many} members. Here $L(\bbR)_{\rm ss}$ denotes the derived group of $L(\bbR)$. Therefore,  Theorem A actually leads to \emph{a finite algorithm} for  pinning  down $\widehat{G(\bbR)}^{\mathrm{d}}$.
From this aspect, we view it as a finiteness result.

%We expect some applications of Theorem A in the theory of automorphic forms via the approach of \cite[Chapter 8]{HP2}, where Huang and Pand\v zi\'c sharpened the results of Langlands \cite{L} and Hotta-Parthasarathy \cite{HoPa} by using the index theorem for the Dirac operator. It is also  conceivable that Theorem A will be helpful for studying the Dirac index polynomial recently introduced by Mehdi, Pand\v zi\'c and Vogan \cite{MPV}.

Among the entire set $\widehat{G(\bbR)}^{\mathrm{d}}$, its scattered part should be viewed as the ``kernel" thus deserves particular attention. Conjecture 5.7 of Salamanca-Riba and Vogan \cite{SV} tries to reduce the classification of unitary representations to the classification of those containing a unitarily small (\emph{u-small} for short) $K(\bbR)$-type.  The reader can refer to Definition 6.1 of \cite{SV} for the key notion of u-small. It will also be recalled under certain assumptions soon. Inspired by this unified conjecture, and based on our previous calculations \cite{DD,D17}, we propose the following.

\medskip
\noindent\textbf{Conjecture B.}
\emph{Take any $\pi$ in the scattered part of $\widehat{G(\bbR)}^{\mathrm{d}}$. Then any $K(
\bbR)$-type of $\pi$ contributing to $H_D(\pi)$ must be unitarily small.}
\medskip

Our next result aims to collect some evidence for the above conjecture. More precisely, we will  investigate the distribution of the spin norm \cite{D} along Vogan pencils \cite{V80}, with particular attention paid to the u-small convex hull. By specifying a Vogan diagram for $\frg_0$, we have chosen a positive root system $\Delta^+(\frg, \frh_f)$. When restricted to $\frt_f$, we have
\begin{equation}\label{restricted-pos-roots}
\Delta^+(\frg, \frt_f)=\Delta^+(\frk, \frt_f) \cup \Delta^+(\frp, \frt_f).
\end{equation}
Let $\rho$ (resp., $\rho_c$) be the half sum of the positive roots in $\Delta^+(\frg, \frt_f)$ (resp., $\Delta^+(\frk, \frt_f)$).
Throughout this paper, the positive root system $\Delta^+(\frk, \frt_f)$ is fixed once for all.

For the remaining part of this section, we assume that {\bf $\frk$ has no center}. Namely, $\frg_0$ is not Hermitian symmetric.
Let $\{\gamma_1, \dots, \gamma_l\}$ be the simple roots of $\Delta^+(\frk, \frt_f)$, and let $\{\varpi_1, \dots, \varpi_l\}$ be the corresponding fundamental weights. We will refer to a $\frk$-type $E_{\mu}$---an irreducible finite-dimensional representation of $\frk$---by its highest weight $\mu=[a_1, \dots, a_l]$, which stands for $a_1\varpi_1+\cdots+a_l\varpi_l$. We denote by $\Lambda$ the weight lattice for $\Delta^+(\frk, \frt_f)$, and collect the dominant weights as $\Lambda^+$. Let $C$ be the dominant Weyl chamber for $\Delta^+(\frk, \frt_f)$, and collect all the non-negative integer combinations of $\gamma_1$, ..., $\gamma_l$ as $\Pi$.

Let
\begin{equation}
R(\Delta(\frp, \frt_f))=\left\{   \sum_{\alpha\in\Delta(\frp, \frt_f)} b_{\alpha} \alpha\mid 0\leq b_{\alpha}\leq 1\right\}.
\end{equation}
This convex set is invariant under $W(\frk, \frt_f)$, and it is  the \emph{u-small convex hull} introduced by Salamanca-Riba and Vogan in \cite{SV}.  We call a $\frk$-type $E_{\mu}$ \emph{u-small} if its highest weight $\mu$ lies in $R(\Delta(\frp, \frt_f))$; otherwise, we would say that $E_{\mu}$ is \emph{u-large}. The notion \emph{spin norm} will be recalled in \eqref{spin-norm}.  Since $\frg_0$ is assumed to be not Hermitian symmetric, the $\frk$-representation $\frp$ is irreducible.
Thus it has a unique highest weight which will be denoted by $\beta$. According to Corollary 3.5 of Vogan \cite{V80}, the $\frk$-types in any infinite-dimensional $(\frg, K(\bbR))$-module must be the  union of \emph{pencils}, which are of the forms
\begin{equation}\label{pencil}
P(\mu):=\{\mu+n\beta|  n\in \bbZ_{\geq 0}\}.
\end{equation}
These objects are illustrated for the $G_{2(2)}$ case in Figure 1 below, where
the shaded region is $R(\Delta(\frp, \frt_f))\cap C$, while the dotted  circles stand for u-small
$\frk$-types. Note that $2\rho_n$, $2\rho_n^{(1)}$ and $2\rho_n^{(2)}$ there are extremal points of the u-small convex hull.

\begin{figure}[H]
\centering
\scalebox{0.9}{\includegraphics{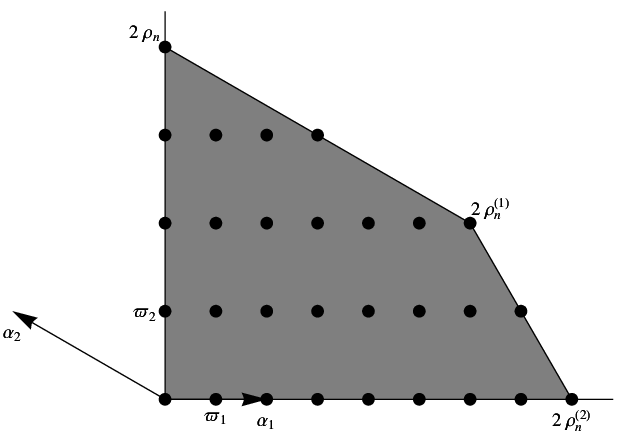}}
\caption{The $G_{2(2)}$ case, where $\beta=3\varpi_1+\varpi_2$.}
\label{Fig-G-USmall}
\end{figure}

\medskip
\noindent\textbf{Theorem C.} \,
\emph{Let $\frg_0$ be on the following list
$$
\frs\frl (2n, \bbR), \, n\geq 2; \quad \frs\frl (2n+1, \bbR); \quad  \frs\frl (n, \bbH), \, n\geq 2;
$$
$$
{\rm E_{6(6)}, \, E_{6(2)}, \, E_{6(-26)}, \, E_{7(7)}, \, E_{7(-5)}, \, E_{8(8)}, \, E_{8(-24)}, \, F_{4(4)}, \, F_{4(-20)}, \, G_{2(2)}}.
$$
The
spin norm increases strictly along any pencil once it goes beyond
the u-small convex hull. Namely,  for any u-large weight $\mu$ such
that $\mu-\beta$ is dominant, we have
\begin{equation}\label{pencil-reduction-step}
\| \mu \|_{\mathrm{spin}} > \|\mu-\beta \|_{\mathrm{spin}}.
\end{equation}}
\medskip

The requirement that  $\mu$  is u-large is key for
\eqref{pencil-reduction-step} to hold. For example, in the $G_{2(2)}$ case, \eqref{pencil-reduction-step} fails for the u-small $\frk$-type $\mu=\beta$. Indeed, in that case, we have $$\|\beta\|_{\rm spin}=\|\rho_c\|<\|0\|_{\rm spin}=\|\rho\|.$$ Moreover, we note  that the unitary dual is known when $\frg_0$ is type A or $G_{2(2)}$, see \cite{Vog86} and \cite{Vog94}.

 Earlier, Theorem 1.1 of \cite{D16}---the counterpart of Theorem C for complex Lie groups---turned out to be very effective in controlling the infinitesimal characters for unitary representations. For instance, in our determination of all the equivalence classes of irreducible unitary representations  with nonvanishing Dirac cohomology for complex $E_6$ \cite{D17}, applying Theorem 1.1 of \cite{D16} has reduced the number of candidate representations in an $s$-family from $124048$ to $3$, see Example 4.1 of \cite{D17}.  Therefore, we expect that Theorem C will improve the efficiency of detecting non-unitarity of irreducible representations of the concerned real reductive Lie groups.

We remark that for other classical real Lie algebras which are not Hermitian symmetric, our method in Section \ref{sec-classical-Lie-algebras} met difficulties. Indeed, we carefully investigated $\frs\frp(p, q)$ with $q\geq p\geq 1$---which should be the easiest case among the remaining ones---but in vain. In this case, the u-small convex hull becomes very complicated, and we did not find an effective way to handle the great many inequalities defining it.

\textbf{When $\frk$ has center}, the $K(\bbR)$-types of an infinite-dimensional irreducible $(\frg, K(\bbR))$-module may not be the union of pencils. Thus there is no motivation to deduce Theorem C in this case. However, we will directly handle the ${\rm Sp}(4, \bbR)$ case in Section \ref{sec-Sp4R}. In particular, it suggests that a neat analogue of Theorem C is not obvious.

The paper is organized as follows. We collect some preliminaries in Section 2, and prove Theorem A in Section 3.
We illustrate our strategy for Theorem C in Section 4. Then we handle the classical Lie algebras for Theorem C in Section 5, and deal with the exceptional ones in Section 6. Section \ref{sec-Sp4R} studies ${\rm Sp}(4, \bbR)$. Finally, we give an algorithm for computing some scattered members in $\widehat{G(\bbR)}^d$ and carry it out for $E_{6(-26)}$ in Section 8.

 Throughout this paper, all the data about root systems are adopted as in Appendix C of Knapp \cite{Kn}.

\emph{Acknowledgements.} The author thanks his thesis adviser Prof.~Huang sincerely for sharing brilliant ideas with him during his PhD study. For instance, Huang suggested the author to pay attention to the u-small convex hull in 2010.  Theorem A grew out of the very helpful lectures given on the 2017 \texttt{atlas} workshop, which was held at University of Utah, July 10--21. We express our sincere gratitude to the \texttt{atlas} mathematicians for their generous support.  We also thank the math department of MIT for offering excellent working conditions. Finally, I am deeply grateful to an anonymous referee for his/her eagle eyes and great patience.

\section{Preliminaries}
This section aims to collect some preliminaries.
\subsection{$\texttt{atlas}$ height and lambda norm}\label{sec-lambda-spin}
We adopt the basic notation $G$, $G(\bbR)$, $\theta$, etc.,  as in the introduction.
In particular, a non-degenerate invariant symmetric bilinear form $B(\cdot, \cdot)$ has been fixed on $\frg$.

We denote by $\Delta(\frg, \frh_f)$ (resp.,  $\Delta(\frg, \frt_f)$) the root system of $\frg$ with respect to $\frh_f$ (resp., $\frt_f$).
The root system of $\frk$ with respect to $\frt_f$ is denoted by $\Delta(\frk, \frt_f)$.
Note that $\Delta(\frg,
\frh_f)$ and $\Delta(\frk, \frt_f)$ are reduced, while $\Delta(\frg, \frt_f)$ is not reduced in general. The Weyl groups for these root systems will be denoted by $W(\frg, \frh_f)$, $W(\frg, \frt_f)$ and $W(\frk, \frt_f)$.

We fix compatible choices of positive roots $\Delta^+(\frg, \frh_f)$ and $\Delta^+(\frg, \frt_f)$ so that a positive root in $\Delta(\frg, \frh_f)$ restricts to a positive root in $\Delta(\frg, \frt_f)$. Recall the decomposition \eqref{restricted-pos-roots}.
Denote by $\rho_n$ the half sum of roots in $\Delta^+(\frp, \frt_f)$. Then $\rho_n=\rho-\rho_c$, and they all live in $i\frt_{f, 0}^*$.

Let us simply refer to a $\frk$-type by its highest weight $\mu$. Choose a positive root system $(\Delta^+)^\prime(\frg, \frh_f)$ making $\mu+2\rho_c$ dominant. Let $\rho^\prime$ be the half sum of roots in $(\Delta^+)^\prime(\frg, \frh_f)$. After \cite{SV}, we put $\lambda_a(\mu)$ as the projection of $\mu+2\rho_c- \rho^\prime$ to the dominant Weyl chamber of $(\Delta^+)^\prime(\frg, \frh_f)$. Then
\begin{equation}\label{spin-norm}
\|\mu\|_{\rm{lambda}}:=\|\lambda_a(\mu)\|.
\end{equation}
Here $\|\cdot\|$ is the norm on $i\frt_{f, 0}^*$ induced from the form $B(\cdot, \cdot)$. It turns out that this number  is  independent of the choice of  $(\Delta^+)^\prime(\frg, \frh_f)$, and it is the \emph{lambda norm} of the $\frk$-type $\mu$ \cite{Vog81}. The \texttt{atlas} \emph{height} of $\mu$ is defined as
\begin{equation}\label{atlas-height}
\sum_{\alpha\in (\Delta^+)^\prime(\frg, \frh_f)}\langle \mu,  \alpha^{\vee}\rangle.
\end{equation}

%For the fixed $\Delta^+(\frk, \frt_f)$, let us enumerate all the compatible choices of positive roots for $\Delta(\frp, \frt_f)$ as
%$$
%(\Delta^+)^{(j)}(\frp, \frt_f), \quad 0\leq j\leq s-1.
%$$
%Here $(\Delta^+)^{(0)}(\frp, \frt_f)=\Delta^+(\frp, \frt_f)$. Denote by $\rho_n^{(j)}$ the half sum of roots in $(\Delta^+)^{(j)}(\frp, \frt_f)$. In particular, $\rho_n^{(0)}=\rho_n$. Now the \emph{spin norm} of the $\frk$-type $\mu$  is defined as
%\begin{equation}\label{spin-norm}
%\|\mu\|_{\rm{spin}}:=\min_{0\leq j\leq s-1} \|\{\mu-\rho_n^{(j)}\}+\rho_c\|,
%\end{equation}
%where $\{\mu-\rho_n^{(j)}\}$ stands for the unique dominant weight to which $\mu-\rho_n^{(j)}$ is conjugate under the action of $W(\frk, \frt_f)$.

\subsection{Dirac cohomology}\label{sec-Dirac}

Fix an orthonormal basis $Z_1, \cdots, Z_n$ of $\frp_0$ with respect to
the inner product induced by the form $B(\cdot, \cdot)$. Let $U(\frg)$ be the
universal enveloping algebra of $\frg$ and let $C(\frp)$ be the
Clifford algebra of $\frp$ with respect to $B$. The Dirac operator
$D\in U(\frg)\otimes C(\frp)$ is defined as
$$D=\sum_{i=1}^{n}\, Z_i \otimes Z_i.$$
It is easy to check that $D$ does not depend on the choice of the
orthonormal basis $Z_i$ and it is $K(\bbR)$-invariant for the diagonal
action of $K(\bbR)$ given by adjoint actions on both factors.

To understand the unitary dual $\widehat{G(\bbR)}$ better, Vogan introduced the notion of Dirac cohomology in 1997 \cite{Vog97}.
Let $\widetilde{K(\bbR)}$ be  the subgroup of $K(\bbR)\times \text{Pin}\,\frp_0$ consisting of all pairs $(k, s)$ such that $\text{Ad}(k)=p(s)$, where $\text{Ad}: K(\bbR)\rightarrow \text{O}(\frp_0)$ is the adjoint action, and $p: \text{Pin}\,\frp_0\rightarrow \text{O}(\frp_0)$ is the pin double covering map. Namely, $\widetilde{K(\bbR)}$ is constructed from the following diagram:
\[
\begin{CD}
\widetilde{K(\bbR)} @>  >  > {\rm Pin}\, \frp_0 \\
@VVV  @VVpV \\
K(\bbR) @>{\rm Ad}>> O(\frp_0)
\end{CD}
\]
Let $S_G$ be a spin module for
$C(\frp)$, then $S_G$ is a $\widetilde{K(\bbR)}$ module.
Let $\pi$ be a
($\frg$, $K(\bbR)$)-module.
Then $D\in U(\frg)\otimes C(\frp)$ acts on $\pi\otimes S_G$, and the Dirac
cohomology of $\pi$ is defined as the $\widetilde{K(\bbR)}$-module
\begin{equation}\label{def-Dirac-cohomology}
H_D(\pi)=\text{Ker}\, D/ (\text{Im} \, D \cap \text{Ker} D).
\end{equation}
Here we note that $\widetilde{K(\bbR)}$ acts  on $\pi$
through $K(\bbR)$ and on $S_G$ through the pin group
$\text{Pin}\,{\frp_0}$. Moreover, since ${\rm Ad}(k) (Z_1), \dots, {\rm Ad}(k) (Z_n)$ is still an orthonormal basis of $\frp_0$, it follows that $D$ is $\widetilde{K(\bbR)}$ invariant. Therefore, ${\rm  Ker} D$, ${\rm Im} D$, and $H_D(X)$ are once again $\widetilde{K(\bbR)}$ modules.

By setting  the linear functionals on $\frt_f$ to be zero on $\fra_f$, we embed $\frt_f^{*}$ as a subspace of $\frh_f^{*}$. The  Vogan conjecture  was proved by Huang and Pand\v zi\'c in Theorem 2.3 of \cite{HP}. Let us recall a slight extension of this result to possibly disconnected Lie groups as follows.

\begin{thm}{\rm (Theorem A \cite{DH})}\label{thm-HP}
Let $\pi$ be an irreducible ($\frg$, $K(\bbR)$)-module.
Assume that the Dirac
cohomology of $\pi$ is nonzero, and let $\gamma\in\frt_f^{*}\subset\frh_f^{*}$ be any highest weight of any $\widetilde{K(\bbR)}$-type  in $H_D(X)$. Then the infinitesimal character $\Lambda$ of $\pi$ is conjugate to
$\gamma+\rho_{c}$ under $W(\frg,\frh_f)$.
\end{thm}

Guaranteed by the above theorem, a necessary condition for $\pi$ to have non-zero Dirac cohomology is that it has real infinitesimal character (cf. Definition 5.4.11 of \cite{Vog81}).

%We care the most about the case that $\pi$ is unitary. Then $D$ is self-adjoint with respect to a natural inner product on $\pi\otimes S_G$, $\text{Ker} D \cap \text{Im} D=0$, and
%\begin{equation}\label{Dirac-unitary}
%H_D(\pi)=\text{Ker}\, D=\text{Ker}\, D^2.
%\end{equation}
%Parthasarathy's Dirac operator inequality \cite{P1, P2} now says that $D^2$ has non-negative eigenvalue on any $\widetilde{K(\bbR)}$-type of $\pi\otimes S_G$. Namely, we have that
%\begin{equation}\label{Dirac-inequality}
%\|\mu\|_{\rm spin}\geq \|\Lambda\|,
%\end{equation}
%where $\mu$ is a highest weight of any $K(\bbR)$-type occurring in $\pi$.
%Moreover, by Theorem 3.5.2 of \cite{HP2}, \eqref{Dirac-inequality} becomes equality on certain $K(\bbR)$-types if and only if $H_D(\pi)$ is non-vanishing.

\subsection{Cohomological induction}
Fix a non-zero element $H\in i\frt_{f, 0}$, then  a $\theta$-stable parabolic subalgebra $\frq= \frl\oplus\fru$ of $\frg$ can be defined as the sum of nonnegative eigenspaces of $\mathrm{ad}(H)$. Here the Levi subalgebra $\frl$ of
$\frq$ is the zero eigenspace of $\mathrm{ad}(H)$, while the
nilradical $\fru$ of $\frq$ is the sum of positive eigenspaces of
$\mathrm{ad}(H)$.  Then it follows from $\theta(H)=H$ that $\frl$, $\fru$ and $\frq$ are all
$\theta$-stable. Set $L(\bbR)=N_{G(\bbR)}(\frq)$.

Let us arrange the positive root systems in a compatible way, that
is, $\Delta(\fru, \frh_f)\subseteq \Delta^{+}(\frg,\frh_f)$ and
set $\Delta^{+}(\frl, \frh_f)=\Delta(\frl,
\frh_f)\cap \Delta^{+}(\frg,\frh_f)$. Let
$\rho^{L}$  denote the half sum of roots in
$\Delta^{+}(\frl,\frh_f)$,  and denote by $\rho(\fru)$ (resp., $\rho(\fru\cap\frp)$) the half sum of roots in $\Delta(\fru,\frh_f)$ (resp., $\Delta(\fru\cap\frp,\frh_f)$). Then
\begin{equation}\label{relations}
\rho=\rho^{L}+\rho(\fru).
\end{equation}

Let $Z$ be an ($\frl$, $L(\bbR)\cap K(\bbR)$)-module. Cohomological induction functors (or Zuckermann functors) attach to $Z$ certain ($\frg, K(\bbR)$)-modules $\caL_j(Z)$ and $\caR^j(Z)$, where $j$ is a nonnegative integer. For a definition, see Chapter 2 of \cite{KV}.
Suppose that $\lambda_L\in \frh_f^*$ is the infinitesimal character of $Z$. We say
$Z$ or $\lambda_L$ is {\it good} or {\it in good range} if
\begin{equation}\label{def-good}
\mathrm{Re}\langle \lambda_L +\rho(\fru),\, \alpha \rangle >
0, \quad \forall \alpha\in \Delta(\fru, \frh_f).
 \end{equation} We say $Z$ or
$\lambda$ is {\it weakly good} if
\begin{equation}\label{def-weakly-good}
\mathrm{Re}\langle \lambda_L +\rho(\fru),\, \alpha \rangle
\geq 0, \quad \forall \alpha\in \Delta(\fru, \frh_f).
\end{equation}

Let us recall a theorem which is mainly due to Vogan.

\begin{thm}\label{thm-Vogan-coho-ind}
{\rm (\cite{Vog84} Theorems 1.2 and 1.3, or \cite{KV} Theorems 0.50 and 0.51)}
Suppose the admissible
 ($\frl$, $L(\bbR)\cap K(\bbR)$)-module $Z$ is weakly good.  Then we have
\begin{itemize}
\item[(i)] $\caL_j(Z)=\caR^j(Z)=0$ for $j\neq S(:=\emph{\text{dim}}\,(\fru\cap\frk))$.
\item[(ii)] $\caL_S(Z)\cong\caR^S(Z)$ as ($\frg$, $K(\bbR)$)-modules.
\item[(iii)]  if $Z$ is irreducible, then $\caL_S(Z)$ is either zero or an
irreducible ($\frg$, $K(\bbR)$)-module with infinitesimal character $\lambda_L+\rho(\fru)$.
\item[(iv)]
if $Z$ is unitary, then $\caL_S(Z)$, if nonzero, is a unitary ($\frg$, $K(\bbR)$)-module.
\item[(v)] if $Z$ is in good range, then $\caL_S(Z)$ is nonzero, and it is unitary if and only if $Z$ is unitary.
\end{itemize}
\end{thm}

Let $\pi$ be an irreducible $(\frg, K(\bbR))$-module with \emph{real} infinitesimal character $\Lambda\in\frh_f^*$ which is dominant for $\Delta^+(\frg, \frh_f)$.
After \cite{Sa}, we say that $\pi$ is {\it strongly regular} if
\begin{equation}\label{def-weakly-good}
\langle \Lambda-\rho,\, \alpha^\vee \rangle
\geq 0, \quad \forall \alpha\in \Delta^+(\frg, \frh_f).
\end{equation}

\begin{thm}\label{thm-SR} \emph{(Salamanca-Riba \cite{Sa})}
Let $\pi$ be a strongly regular irreducible $(\frg, K(\bbR))$-module. Then $\pi$ is unitary if and only if it is a $A_{\frq}(\lambda)$ module in the good range.
\end{thm}

Recall that Dirac cohomology of  cohomologically induced modules has been determined whenever the inducing modules are weakly good \cite{DH}.

\begin{thm}{\rm (Theorem B of \cite{DH})}\label{thm-DH}
 Suppose that the irreducible unitary $(\frl, L(\bbR)\cap K(\bbR))$-module $Z$ has a real infinitesimal character $\lambda_L\in i\frt_{f, 0}^*$ which is weakly good. Then there is a $\widetilde{K(\bbR)}$-module isomorphism
\begin{equation}\label{Dirac-coho}
H_D(\caL_S(Z)) \cong  \caL_S^{\widetilde{K(\bbR)}} (H_D(Z)\otimes \bbC_{-\rho(\fru\cap\frp)}).
\end{equation}
\end{thm}

\subsection{Cohomological induction in \texttt{atlas}}

Let us recall necessary notation from \cite{ALTV} regarding the Langlands parameters in the software \texttt{atlas} \cite{At}. Let $H$ be a \emph{maximal torus} of $G$. That is, $H$ is a maximal connected abelian subgroup of $G$ consisting of diagonalizable matrices. Note that $H$ is complex connected reductive algebraic. Its\emph{ character lattice} is the group of algebraic homomorphisms
$$
X^*:={\rm Hom}_{\rm alg} (H, \bbC^{\times}).
$$
Choose a Borel subgroup $B\supset H$.
In \texttt{atlas}, an irreducible $(\frg, K(\bbR))$-module $\pi$ is parameterized by a \emph{final} parameter $p=(x, \lambda, \nu)$ via the Langlands classification, where $x$ is a $K$-orbit of the Borel variety $G/B$, $\lambda \in X^*+\rho$ and $\nu \in (X^*)^{-\theta}\otimes_{\bbZ}\bbC$. For the parameter $p$ to be final, there are further requirements on $x$, $\lambda$ and $\nu$. We refer the reader to \cite{ALTV} for a rigorous definition. In such a case, the infinitesimal character of $\pi$ is
\begin{equation}\label{inf-char}
\frac{1}{2}(1+\theta)\lambda +\nu \in\frh^*.
\end{equation}
Note that the Cartan involution $\theta$ now becomes $\theta_x$---the involution of $x$, which is given by the command \texttt{involution(x)} in \texttt{atlas}.

The following result is taken from Paul's lecture \cite{Paul}. It tells us how  to do cohomological induction in \texttt{atlas}.

\begin{thm}\label{thm-Vogan} \emph{(Vogan \cite{Vog84})}
 Let $p=(x, \lambda, \nu)$ be the \texttt{atlas} parameter of an irreducible $(\frg, K(\bbR))$-module $\pi$.
Let $S$ be the support of $x$, and $\frq(x)$ be the $\theta$-stable parabolic subalgebra given by the pair $(S, x)$, with Levi factor $L(\bbR)$.
Then  $\pi$ is cohomologically induced, in the weakly good range,
from an irreducible $(\frl, L(\bbR)\cap K(\bbR))$-module $\pi_L$ with parameter $p_L$.
\end{thm}

Note that the support of a KGB element $x$ is given by the command \texttt{support(x)} in \texttt{atlas}. The  parameter  $p_L=(y, \lambda^\prime, \nu^\prime)$ can be easily obtained as follows:
$y$ is the KGB element of $L(\bbR)$ corresponding to the KGB element $x$ of $G(\bbR)$, $\lambda^\prime=\lambda-\rho(\fru)$ and $\nu^\prime=\nu$.

\begin{example}\label{exam-coho-ind}
Let us illustrate Theorem \ref{thm-Vogan} via a representation of the  group $\texttt{F4\_B4}$ in \texttt{atlas}, whose Lie algebra is ${\rm FI}=F_{4(4)}$ (see Section \ref{sec-FI}). To save space, certain outputs have been omitted here.
\begin{verbatim}
atlas> G:F4_B4
atlas> set p=parameter(KGB(G,1), [1,1,0,1], [0,0,0,0])
atlas> set (P, pL)=reduce_good_range(p)
atlas> P
Value: ([],KGB element #1)
atlas> pL
Value: final parameter(x=0,lambda=[0,0,-1,0]/1,nu=[0,0,0,0]/1)
atlas> rho_u(P)
Value: [ 1, 1, 1, 1 ]/1
atlas> theta_induce_irreducible(pL, G)=p
Value: true
atlas> goodness(pL,G)
Value: "Weakly good"
\end{verbatim}
The \texttt{pL} above is the parameter of the inducing module. The last output says that it is weakly good.\hfill\qed
\end{example}

%Using the knowledge of the functor $\caL_S^{\widetilde{K(\bbR)}}$ stated after Corollary 5.85 of \cite{KV}, we have the following.
%
%\begin{cor}\label{cor-thm-DH}
%Under the setting of Theorem \ref{thm-DH}, we have that $H_D(\caL_S(Z))$ is non-zero if and only if $H_D(Z)$ is non-zero and that there exists a highest weight $\gamma_L$ in $H_D(Z)$ such that $\gamma_L +\rho(\fru\cap\frp)$ is $\Delta^+(\frk, \frt_f)$ dominant.
%\end{cor}

%Now assume further that $H$ is \emph{defined over $\bbR$} (with respect to $\sigma$). That is, $\sigma(H)=H$. Then we put $H(\bbR)=H(\bbR, \sigma)=H^{\sigma}$, put $T(\bbR)=H(\bbR)^{\theta}$. Write
%$$
%\frh_0=\frt_0+\fra_0
%$$
%for the decomposition of the real Lie algebra of $H(\bbR)$ into $+1$ and $-1$ eigenspaces of $\theta$. Put $A=\exp(\fra_0)$. The group $A$ is isomorphic to its Lie algebra $\fra_0$, and
%$$
%H(\bbR)\cong T(\bbR) A.
%$$
%Then, as shown in Proposition 4.3 of \cite{ALTV}, characters of $H(\bbR)$ can be parameterized by
%$\gamma=(\lambda, \nu)\in \widehat{T(\bbR)}\times \widehat{A}$, where
%$$
%\lambda\in\widehat{T(\bbR)}\cong {\rm Hom}_{\rm alg}(T, \bbC^{\times})\cong X^*/(1-\theta)X^*
%$$
%and
%$$
%\nu\in\widehat{A}\cong\fra^*\cong (X^*)^{-\theta}\otimes_{\bbZ}\bbC.
%$$

\section{Proof of Theorem A}

In this section, we continue to let $G$, $\sigma$, $\theta$, etc., be as in the introduction.
We fix a maximal torus $H$ of $G$ which is \emph{defined over} $\bbR$ with respect to $\sigma$.  That is, $\sigma(H)=H$.
Whenever a Borel subgroup $B\supset H$ is fixed,  we will have a set of positive roots $\Delta^+(\frg, \frh)$. Denote by $\rho(G)$ the half sum of roots in $\Delta^+(\frg, \frh)$. Of course $\|\rho(G)\|$ is independent of the choice of $B$. Now let us prepare two propositions.

\begin{prop}\label{prop-nu-bound}
Let $\pi$ be an irreducible unitary $(\frg, K(\bbR))$-module with \texttt{atlas} parameter $(x, \lambda, \nu)$. Assume that $\pi$ has real infinitesimal character. Then \begin{equation}\label{nu-bound}
\|\nu\| \leq \|\rho(G)\|.
\end{equation}
\end{prop}

The spherical case of the above proposition was due to Helgason and Johnson \cite{HJ}. In general, one can refer to Theorem 5.2 of Chapter IV of Borel and Wallach \cite{BW}.
Put
\begin{equation}\label{N}
N:=\max_{\alpha\in\Delta^+(\frg, \frh)}\frac{2\|\rho(G)\|}{\|\alpha\|}.
\end{equation}

\begin{prop}\label{prop-lambda-bound}
Let $\pi$ be any irreducible unitary $(\frg, K(\bbR))$-module with real infinitesimal character. Let $(x, \lambda, \nu)$ be the \texttt{atlas} parameter of $\pi$.  Then either  $\|\frac{\lambda+\theta\lambda}{2}\|$ is upper bounded by $N \|\rho(G)\|$, or $\pi$ is cohomologically induced from an irreducible unitary Harish-Chandra module $\pi_{L(\bbR)}$ of $L(\bbR)$ which is in the good range. Here $L(\bbR)$ is a proper $\theta$-stable parabolic subgroup of $G(\bbR)$.
\end{prop}
\begin{proof}
It suffices to show that whenever $\|\frac{\lambda+\theta\lambda}{2}\|>N \|\rho(G)\|$, then
$\pi$ must be cohomologically induced from some irreducible unitary  Harish-Chandra module $\pi_{L(\bbR)}$ of $L(\bbR)$ which is  in the
good range. Here $L(\bbR)$ is certain proper $\theta$-stable Levi subgroup of $G(\bbR)$.

Since $\|\nu\|\leq \|\rho(G)\|$ by Proposition \ref{prop-nu-bound}, we have that
\begin{equation}\label{root-for-nu}
|\langle \nu, \alpha^\vee\rangle | \leq  N
\end{equation}
for every root $\alpha$. Choose a Borel subgroup $B$ of $G$ making $\frac{\lambda+\theta\lambda}{2}$ weakly dominant. That is,
$$
\langle \frac{1}{2}(\lambda+\theta\lambda), \beta^\vee\rangle \geq 0
$$
for every positive root $\beta$. Let $\{\zeta_1, \dots, \zeta_l\}$ be the fundamental weights for $\Delta^+(\frg, \frh)$.
Then
$$
\frac{\lambda+\theta\lambda}{2}=\lambda_1\zeta_1 + \cdots + \lambda_l\zeta_l,
$$
where $\lambda_i\geq 0$.

Let $\frl$ be the $\theta$-stable Levi subalgebra of $\frg$ \emph{generated} by all those
positive roots $\beta$ such that
\begin{equation}\label{beta-for-lambda}
\langle
\frac{1}{2}(\lambda+\theta\lambda), \beta^\vee\rangle \le N.
\end{equation}
We \emph{claim} that if $\|\frac{\lambda+\theta\lambda}{2}\|> N\|\rho(G)\|$, then $\frl$ must\emph{ proper}. Indeed, in such a case, there must exist $1\leq i_0\leq l$ such that $\lambda_{i_0}>N$.
Let $\beta_{i_0}$ be the corresponding simple root. Then
\begin{equation}
\langle
\frac{1}{2}(\lambda+\theta\lambda), \beta_{i_0}^\vee\rangle=\lambda_{i_0} > N.
\end{equation}
The root $\beta_0$ is distinct from any of the $\beta$ described in \eqref{beta-for-lambda}. Thus the claim holds.

Let $\frq=\frl+\fru$ be the corresponding $\theta$-stable parabolic subalgebra of $\frg$. If $\gamma$ is
any root in $\fru$, then
\begin{equation}
\langle \frac{1}{2}(\lambda+\theta\lambda), \gamma^\vee\rangle > N.
\end{equation}
Thus by \eqref{root-for-nu}, we have
\begin{equation}
\langle \frac{1}{2}(\lambda+\theta\lambda) +\nu,  \gamma^\vee\rangle > 0.
\end{equation}
Therefore cohomological induction from $\pi_{L(\bbR)}$ by $\frq$ to $G(\bbR)$ (in the way analogous to Theorem \ref{thm-Vogan}) is in the good range, and Theorem \ref{thm-Vogan-coho-ind} says that $\pi_{L(\bbR)}$ must be irreducible and unitary.
\end{proof}
\begin{rmk}\label{rmk-prop-lambda-bound}
The above proof  was kindly told to us by Professor Vogan.
\end{rmk}

Now we are ready to prove Theorem A.

\medskip
\noindent \emph{Proof of Theorem A.}\quad   Let $\pi$ be an irreducible unitary $(\frg, K(\bbR))$-module with non-zero Dirac cohomology. Let $(x, \lambda, \nu)$ be an \texttt{atlas} parameter of $\pi$. By Theorem \ref{thm-HP}, the infinitesimal character \eqref{inf-char} of $\pi$ must be real. There are two cases.

The first case is that $\pi$ is cohomologically induced from an irreducible unitary  Harish-Chandra module $\pi_{L(\bbR)}$ of $L(\bbR)$ which is in the good range, where $L(\bbR)$ is a proper $\theta$-stable Levi subgroup of $G(\bbR)$. In this case, Theorem \ref{thm-DH} says that $\pi_{L(\bbR)}$ must be a member of $\widehat{L(\bbR)}^{\mathrm{d}}$ and we are done.

Otherwise, by Proposition \ref{prop-lambda-bound}, we conclude that $\|\frac{\lambda+\theta\lambda}{2}\|\leq N\|\rho(G)\|$. Note that $\|\nu\|\leq \|\rho(G)\|$ since $\pi$ is assumed to be unitary, see Proposition \ref{prop-nu-bound}. Therefore, the infinitesimal character
\begin{equation}\label{inf-char-atlas}
\frac{1}{2}(1+\theta)\lambda+\nu\in\frh^*
\end{equation}
of $\pi$ is upper bounded by
\begin{equation}\label{Lambda-bound}
(N^2+1)\|\rho(G)\|^2=\left(\max_{\alpha\in\Delta^+(\frg, \frh)}\frac{4\|\rho(G)\|^2}{\|\alpha\|^2}+1\right)\|\rho(G)\|^2.
\end{equation}
Let $\Lambda\in\frh_f^*$ be the conjugation of the above element to $\frh_f^*$. Since $H_D(\pi)$ is assumed to be non-zero, Theorem \ref{thm-HP} tells us that
\begin{equation}\label{inf-char-classical}
\Lambda=w(\gamma_G +\rho_c)
\end{equation}
for some $w\in W(\frg, \frh_f)$ and for some $\widetilde{K(\bbR)}$ highest weight $\gamma_G$ of $\pi\otimes S_G$. Note that $\gamma_G+\rho_c$ in \eqref{inf-char-classical} lives in a discrete set, then so do $\Lambda$ in \eqref{inf-char-classical} and the element in \eqref{inf-char-atlas}.
Being bounded and being discrete simultaneously allow us to conclude that the infinitesimal character \eqref{inf-char-atlas} of $\pi$ has finitely many choices. Since there are finitely many irreducible Harish-Chandra modules with a given infinitesimal character, it follows directly that $\pi$ has at most finitely many choices. This handles the remaining cases, and the proof finishes. \hfill\qed

The above proof actually gives us a method to exhaust all the possible infinitesimal characters of the representations coming from the scattered part of $\widehat{G(\bbR)}^{\rm d}$. We will develop this into an explicit algorithm in Section \ref{sec-algorithm}.

\section{Strategy for Theorem C}
This section aims to explain our strategy for Theorem C.
We continue with the setting of the introduction and assume that $\frk$ \textbf{has no center}. Recall the decomposition \eqref{restricted-pos-roots}, and recall that $C$ is the dominant Weyl chamber corresponding to $\Delta^+(\frk, \frt_f)$.  Denote by $C_{\frg}(i\frt_{f, 0}^*)$  the closed Weyl chamber corresponding to $\Delta^+(\frg, \frt_f)$. Then $C_{\frg}(i\frt_{f, 0}^*)$ is contained in $C$. Define
\begin{equation}
W(\frg, \frt_f)^1=\{w\in W(\frg, \frt_f) \mid w(C_{\frg}(i\frt_{f, 0}^*))\subseteq C\}.
\end{equation}
It is well-known that the multiplication map gives a bijection from $W(\frk, \frt_f)\times W(\frg, \frt_f)^1$ onto $W(\frg, \frt_f)$, see Kostant \cite{K}. Then
$$
\{w\Delta^+(\frp, \frt_f)\mid w\in W(\frg, \frt_f)^1\}
$$
are exactly all the choices of positive roots systems for $\Delta(\frp, \frt_f)$ which are compatible with $\Delta^+(\frk, \frt_f)$. Let us enumerate the elements of $W(\frg, \frt_f)^1$
as $w^{(0)}=e, w^{(1)}, \dots, w^{(s-1)}$. For $\quad 0\leq j\leq s-1$, put
$$
(\Delta^+)^{(j)}(\frp, \frt_f)=w^{(j)}\Delta^+(\frp, \frt_f), \quad (\Delta^+)^{(j)}(\frg, \frt_f)= \Delta^+(\frk, \frt_f)\cup (\Delta^+)^{(j)}(\frp, \frt_f).
$$
Note that  $(\Delta^+)^{(0)}(\frp, \frt_f)=\Delta^+(\frp, \frt_f)$ and $(\Delta^+)^{(0)}(\frg, \frt_f)= \Delta^+(\frg, \frt_f)$. Denote by $\rho_n^{(j)}$ the half sum of the positive roots in $(\Delta^+)^{(j)}(\frp, \frt_f)$.
Then $\rho_n^{(0)}=\rho_n$, and we have
$$
w^{(j)}\rho=\rho_c+\rho_n^{(j)}, \quad 0\leq j\leq s-1.
$$
Recall from Lemma 9.3.2 of \cite{W} that the spin module decomposes into the following $\frk$-types:
$$
S_G=\bigoplus_{j=0}^{s-1} 2^{[l_0/2]}  E_{\rho_n^{(j)}},
$$
where $l_0=\dim_{\bbC} \fra_f$. Now as in \cite{D}, the spin norm of the $\frk$-type $E_{\mu}$ is defined to be
\begin{equation}\label{spin-norm}
\|\mu\|_{\rm{spin}}:=\min_{0\leq j\leq s-1} \|\{\mu-\rho_n^{(j)}\}+\rho_c\|.
\end{equation}
Here $\{\mu-\rho_n^{(j)}\}$ denotes the unique dominant weight to which $\mu-\rho_n^{(j)}$ is conjugate under the action of $W(\frk, \frt_f)$. For instance, $\{-\rho_c\}=\rho_c$. Note that the lowest weights of $S_G$ (without multiplicity) are precisely $-\rho_n^{(j)}$,  $0\leq j\leq s-1$. Thus we have
$$
\|0\|_{\rm{spin}}=\|\rho\|, \quad \|\rho_n^{(j)}\|_{\rm{spin}}=\|\rho_c\|, \quad 0\leq j\leq s-1.
$$

Since $G(\bbR)$ is in the Harish-Chandra class, we may and we will define the lambda norm (resp., spin norm) of a $K(\bbR)$-type as
the lambda norm (resp., spin norm) of any of its highest weight. Now a $K(\bbR)$-type of a $(\frg, K(\bbR))$-module $\pi$ is called a \emph{lambda-lowest} (resp., \emph{spin-lowest}) \emph{$K(\bbR)$-type} if its lambda norm (resp. spin norm) attains the minimum among all the $K(\bbR)$-types of $\pi$.

Now let us prepare a few lemmas about the  u-small convex hull.

\begin{lemma}\label{half-u-small} The convex set
$R(\frac{1}{2}\Delta(\frp,\frt_f))$ is the convex hull of
all the extremal weights of the spin module $S_G$.
\end{lemma}
\begin{proof}
Since any extremal weight of $S_G$ must belong to
$R(\frac{1}{2}\Delta(\frp,\frt_f))$, that LHS contains
RHS is obvious. Since
$R(\frac{1}{2}\Delta(\frp,\frt_f))$ is a convex set
invariant under $W(\frk, \frt_f)$, to verify that LHS is contained in RHS, it suffices to show that all the extremal points of
$R(\frac{1}{2}\Delta(\frp,\frt_f))$ which are dominant
for $\Delta^{+}(\frk,\frt_f)$ belong to the RHS. However,
these points are exactly all the highest weights (without
multiplicity) of $S_G$, and we are done.
\end{proof}

\begin{lemma}\label{real-RS-bound} For any point $\mu\in
R(\frac{1}{2}\Delta(\frp,\frt_f))$, we always have
$$\|\mu+\rho_c\|\leq\|\rho\|.$$
\end{lemma}
\begin{proof}
By Lemma \ref{half-u-small}, we need only to consider the case that
$\mu$ is an extremal weight of $S_G$. Since $\rho_c$ is dominant
for $\Delta^{+}(\frk,\frt_f)$, to have an upper bound for
$\|\mu+\rho_c\|$, one can further focus on the case that
$\mu$ is a highest weight of the spin module $S_G$, say $\mu=w^{(j)}\rho-\rho_c$.
Then $\mu+\rho_c=w^{(j)}\rho$ and the conclusion is now obvious.
\end{proof}

\begin{prop}\label{prop-spin-norm-upper-bound-real} Let $E_{\mu}$ be any
u-small $\frk$-type with highest weight $\mu$. Then
$$\|\rho_c\|\leq \|\mu\|_{\mathrm{spin}}\leq \|\rho\|.$$
\end{prop}
\begin{proof}
The first inequality is obvious from \eqref{spin-norm}.
We note that $\|\mu\|_{\mathrm{spin}}=\|\rho_c\|$ if
and only if $\mu$ is a highest weight of the spin module $S_G$. For the second inequality, we use Theorem 6.7(f) of \cite{SV}, which says that there is a positive system
$(\Delta^{+})^{(j)}(\frg,\frt_f)$ containing
$\Delta^{+}(\frk,\frt_f)$ such that
$$\mu=\sum_{\beta\in(\Delta^{+})^{(j)}(\frp,\frt_f)} c_{\beta}\beta, \quad  c_{\beta}\in [0,1].$$
Note that
$$\|\mu\|_{\mathrm{spin}}\leq \|\{\mu-\rho_n^{(j)}\}+\rho_c\|.$$
Since $\rho_n^{(j)}$ is the half sum of the positive roots in $(\Delta^{+})^{(j)}(\frp,\frt_f)$, we have that $\mu-\rho_n^{(j)}\in
R(\frac{1}{2}\Delta(\frp,\frt_f))$, which is $W(\frk, \frt_f)$
invariant. Hence $\{\mu-\rho_n^{(j)}\}\in
R(\frac{1}{2}\Delta(\frp,\frt_f))$. Therefore, by Lemma \ref{real-RS-bound}, we have
$$\| \{\mu-\rho_n^{(j)}\}+\rho_c\|\leq \|\rho\|.$$
\end{proof}

To compute the u-small $\frk$-types more effectively, let us explicitly
write down  Theorem 6.7(d) of \cite{SV} under the current setting. We prepare a bit more notation.
Let $\{\alpha_1, \dots, \alpha_l\}$ be the simple roots for $\Delta^+(\frg, \frt_f)$, with
$\{\xi_1, \dots, \xi_l\}$ the corresponding fundamental weights.

\begin{lemma}\label{hyperplane-construction}
Any $\mu\in\Lambda^+$ is u-small if and only if $\langle\mu+2\rho_c, w^{(j)}\xi_i
\rangle\leq 2\langle\rho, \xi_i
\rangle$, $1\leq i\leq l$, $0\leq j\leq s-1$.
\end{lemma}
\begin{proof}
Note that $w^{(j)} \alpha_i$, $1\leq i\leq l$, are  the simple roots of $(\Delta^+)^{(j)}(\frg, \frt_f)$, and that
$w^{(j)} \xi_i$, $1\leq i\leq l$, are the corresponding fundamental weights. Thus by Theorem 6.7(d) of \cite{SV}, $\mu$ is u-small if and only if $\langle\mu-2\rho_n^{(j)}, w^{(j)}\xi_i \rangle\leq 0$, which is further equivalent to that
$$
\langle\mu+2\rho_c-2w^{(j)}\rho, w^{(j)}\xi_i \rangle\leq 0, \quad 1\leq i\leq l, \quad 0\leq j\leq s-1.
$$
Since $\langle w^{(j)}\rho, w^{(j)}\xi_i \rangle= \langle \rho, \xi_i \rangle$, the desired description follows.
\end{proof}

Now let us explain our strategy for Theorem C. For $0\leq j\leq s-1$, put
\begin{equation}\label{Delta-mu-j}
\Delta(\mu, j)=\{\mu-\rho_n^{(j)}\}-\{\mu-\beta-\rho_n^{(j)}\}.
\end{equation}
Then
\begin{align*}
\|\{\mu-\rho_n^{(j)}\}+\rho_c\|^2
&= \|\rho_c\|^2 + \|\{ \mu-\rho_n^{(j)}  \} \|^2 + 2 \langle\rho_c, \{\mu-\rho_n^{(j)}\}\rangle\\
&= \|\rho_c\|^2 + \| \mu-\rho_n^{(j)}  \|^2 + 2 \langle\rho_c, \{\mu-\rho_n^{(j)}\}\rangle.
\end{align*}
Similarly,
$$
\|\{\mu-\beta-\rho_n^{(j)}\}+\rho_c\|^2= \|\rho_c\|^2 + \| \mu-\beta-\rho_n^{(j)}  \|^2 + 2 \langle\rho_c, \{\mu-\beta-\rho_n^{(j)}\}\rangle.
$$
Thus
\begin{equation}\label{Diff-j}
\|\{\mu-\rho_n^{(j)}\}+\rho_c\|^2- \|\{\mu-\beta-\rho_n^{(j)}\}+\rho_c\|^2=\rm{I}+ \rm{II}.
\end{equation}
where \begin{equation}\label{I}
{\rm I}=2 \langle\rho_c, \Delta(\mu, j)\rangle,
 \end{equation}
 and
 \begin{equation}\label{II}
 {\rm II}=\| \mu-\rho_n^{(j)}  \|^2- \| \mu-\beta-\rho_n^{(j)}  \|^2.
 \end{equation}
The term ${\rm II }$ is relatively easier to analyze, while the term ${\rm I }$ is subtle. Indeed, note firstly that $\Delta(\mu, j)$ lies in $\Lambda$, the weight lattice for $\Delta(\frk, \frt_f)$.  Secondly, let $w\in W(\frk, \frt_f)$ be such that
\begin{equation}\label{mu-beta-rhon-w}
\{\mu-\beta-\rho_n^{(j)}\}=w(\mu-\beta-\rho_n^{(j)}).
\end{equation}
Then we have
\begin{align*}
\{\mu-\rho_n^{(j)}\}-\{\mu-\beta-\rho_n^{(j)}\} &=(\{\mu-\rho_n^{(j)}\}-w(\mu-\rho_n^{(j)}))
+(w(\mu-\rho_n^{(j)})-w(\mu-\beta-\rho_n^{(j)}))\\
&=(\{\mu-\rho_n^{(j)}\}-w(\mu-\rho_n^{(j)})) +w\beta.
\end{align*}
By the
highest weight theorem (see e.g. Theorem 5.5 of \cite{Kn}), the first term above lies in $\Pi$. Thus its inner product with $\rho_c$ is always non-negative. Unlike the complex case studied in \cite{D16}, we may no longer conclude that $w\beta\in
\Delta^{+}(\frk, \frt_f)$. Indeed, the root $\beta$ itself may not lives in $\Delta(\frk, \frt_f)$. Therefore, we can \emph{not} always have that ${\rm I}>0$. However, by the same proof for \eqref{I-parabolic-lower-bound} below, we have that
$$
\langle \rho_c, w\beta\rangle\geq
\langle \rho_c, w_0\beta\rangle=\langle w_0^{-1}\rho_c, \beta\rangle=\langle w_0\rho_c, \beta\rangle=
-\langle \rho_c, \beta\rangle.
$$
Here $w_0$ is the longest element of $W(\frk, \frt_f)$. Therefore, there is a naive lower bound for I. Namely,
\begin{equation}\label{I-naive-lower-bound}
{\rm I} \geq -2 \langle\rho_c, \beta\rangle.
\end{equation}

In actual calculation when $\frg_0$ is exceptional, we will adopt the parabolic subgroups of $W(\frk, \frt_f)$ to sharpen the above lower bound. To be more precise, let $W_k$ be the subgroup of $W(\frk, \frt_f)$ generated by $s_{\gamma_1}, \dots, \widehat{s_{\gamma_k}}, \dots, s_{\gamma_l}$, where  the hat means the $k$-th element is omitted. Let $w_{0, k}$ be the longest element of $W_k$. Then we have
\begin{equation}\label{I-parabolic-lower-bound}
\langle\rho_c, w\beta\rangle \geq  \langle\rho_c, w_{0, k}\beta\rangle, \quad \forall w\in W_k.
\end{equation}
Indeed, let $w^{-1}=s_{\delta_1}\cdots s_{\delta_n}$ be a reduced decomposition of $w^{-1}$ into simple reflections, where $\delta_i\in\{\gamma_1, \dots, \gamma_l\}$.
Then by Lemma 5.5 of \cite{DH11}, we have
\begin{equation}\label{winv-beta}
\rho_c- w^{-1}\rho_c =\sum_{k=1}^{n} \langle \rho_c,
\check{\delta_{k}}\rangle\, s_{\delta_1}s_{\delta_2}\cdots
s_{\delta_{k-1}}(\delta_{k})=\sum_{k=1}^{n} s_{\delta_1}s_{\delta_2}\cdots
s_{\delta_{k-1}}(\delta_{k}),
\end{equation}
where $\check{\delta_{k}}$ is the dual root of $\delta_{k}$ and each $s_{\delta_1}s_{\delta_2}\cdots
s_{\delta_{k-1}}(\delta_{k})$ is a positive root. Since $\beta$ is a dominant weight and each element $w$ of $W_k$ can be extended to $w_{0,k}$ by adding simple reflections,
now \eqref{I-parabolic-lower-bound} follows from \eqref{winv-beta} and that
$$
\langle\rho_c, w\beta\rangle=\langle  w^{-1}\rho_c, \beta\rangle
=\langle  \rho_c, \beta\rangle- \langle  \rho_c-w^{-1}\rho_c, \beta\rangle.
$$
Eventually, we will be able to find positive integers $N_k$ ($1\leq k \leq l$)  such that  \eqref{pencil-reduction-step} holds for any dominant weight $\mu=[a_1, \dots, a_l]$ whenever $a_1\geq N_1$, \emph{or} $a_2\geq N_2$, ..., \emph{or} $a_l\geq N_l$. Therefore,  it remains to check \eqref{pencil-reduction-step} for those u-large $\frk$-types in the following \emph{finite} set
$$
\{\mu=[a_1, \dots, a_l] \mid 0\leq a_k\leq N_k-1, 1\leq k\leq l; \, \mu-\beta \mbox{ is dominant}\}.
$$
And we can leave the latter job to a computer.

The case $E_{6(6)}$ considered in \S\ref{sec-EI} will be a typical example to illustrate our strategy for exceptional Lie algebras.

\section{Classical Lie algebras}\label{sec-classical-Lie-algebras}
This section aims to deal with the classical Lie algebras for Theorem C.

\subsection{$\frs\frl(2n,\bbR)$}
This subsection aims to handle $\frs\frl(2n,\bbR)$ ($n\geq 2$), whose Vogan diagram is presented in Fig.~\ref{Fig-sl-2n-Vogan}. In this case, one calculates that
$$
\Delta^+(\frk, \frt_f)=\{e_i\pm e_j \mid 1\leq i<j\leq n\}, \quad
\Delta^+(\frp, \frt_f)=\Delta^+(\frk, \frt_f)\cup \{2e_i\mid 1\leq i\leq n\}.
$$
We have $W(\frg, \frt_f)^1=\{e, s_{2e_n}\}$, $\beta=2e_1$,
$$
\rho_n=n e_1+(n-1)e_2+\cdots+2e_{n-1}+e_n, \quad \rho_n^{(1)}=n e_1+(n-1)e_2+\cdots+2e_{n-1}-e_n.
$$
Moreover, $\xi_i=e_1+\cdots+e_i$, $1\leq i\leq n$; $\varpi_i=\xi_i$, $1\leq i\leq n-2$,
$$
\varpi_{n-1}=\frac{1}{2}(e_1+\cdots+e_{n-1}-e_n), \quad
\varpi_{n}=\frac{1}{2}(e_1+\cdots+e_{n-1}+e_n).
$$

\begin{figure}[H]
\centering
\scalebox{0.6}{\includegraphics{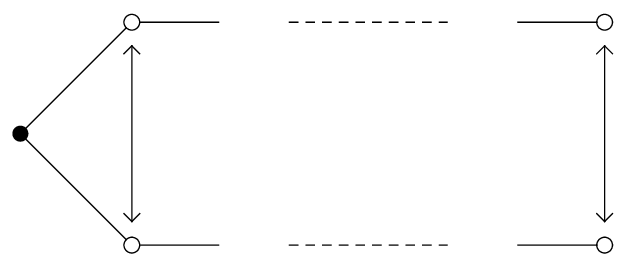}}
\caption{The Vogan diagram for $\frs\frl(2n,\bbR)$}
\label{Fig-sl-2n-Vogan}
\end{figure}

Let $\mu=[m_1, \dots, m_n]$ be a dominant weight. Then $\mu=(a_1, \dots, a_n)$, where
$$
a_i=\sum_{j=i}^{n-2}m_j + \frac{m_{n-1}+m_n}{2}, 1\leq i\leq n-2;\,
a_{n-1}=\frac{m_{n-1}+m_n}{2};\, a_n=\frac{-m_{n-1}+m_n}{2}.
$$
Note that $a_1\geq \dots \geq a_n$ are simultaneously integers or half-integers.

Guided by Lemma \ref{hyperplane-construction}, we calculate that the $\frk$-type $\mu$ is u-small if and only if
\begin{equation}\label{sl-2nR-u-small}
a_1+\cdots+a_k\leq 2nk-k^2+k, \quad 1\leq k\leq n;\quad a_1+\cdots+a_{n-1}-a_n\leq n^2+n.
\end{equation}
We \emph{claim} that $\mu$ is u-small when $a_1\leq n+1$. Indeed, in this case, we would have that
$$
a_1+\cdots+a_k\leq nk+k\leq 2nk-k^2+k,
$$
and that
$$
a_1+\cdots+a_{n-1}-a_n\leq n^2+n.
$$
Therefore, the claim follows from \eqref{sl-2nR-u-small}.

Now let us assume that $\mu$ is u-large and verify \eqref{pencil-reduction-step}.
As mentioned above, we have $a_1\geq n+\frac{3}{2}$.  Firstly, set $j=0$. Then
\begin{align*}
\mu-\rho_n &=(a_1-n,  a_2-(n-1),  \dots, a_n-1),\\
\mu-\beta-\rho_n&=(a_1-n-2,  a_2-(n-1),  \dots, a_n-1),
\end{align*}
and
$$
\|\mu-\rho_n\|^2-\|\mu-\rho_n-\beta\|^2=4(a_1-n-1)>0.
$$
Since $a_1\geq n+\frac{3}{2}$, we always have $a_1-n>|a_1-n-2|$.
Let $B$ be the multi-set consisting of $|a_2-(n-1)|$, ..., $|a_n-1|$. Let $b_1\geq \cdots\geq b_{n}$ be the re-ordering of $a_1-n$ and $B$. Let $c_1\geq \cdots\geq c_{n}$ be the re-ordering of $|a_1-n-2|$ and $B$. Recall that $W(\frk, \frt_f)$ consists of permutations as well as all
\emph{even} sign changes. Thus
\begin{align*}
\{\mu-\rho_n\} &=(b_1,   \dots, b_{n-1}, *),\\
\{\mu-\beta-\rho_n\}&=(c_1,\dots, c_{n-1}, *),
\end{align*}
where the last entries are omitted since they do not affect the parings with $\rho_c=(n-1,\dots, 1,0)$. Since $b_1\geq c_1$, $\dots$, $b_{n-1}\geq c_{n-1}$, we have
$$
2\langle \rho_c, \{\mu-\rho_n\}-\{\mu-\beta-\rho_n\} \rangle \geq 0.
$$
Therefore, the LHS of \eqref{Diff-j} is positive for $j=0$. The same proof shows that
the LHS of \eqref{Diff-j} is positive for $j=1$ as well. Thus \eqref{pencil-reduction-step} holds.

\subsection{$\frs\frl(2n+1,\bbR)$}

This subsection aims to handle $\frs\frl(2n+1,\bbR)$, whose Vogan diagram is presented in Fig.~\ref{Fig-sl-2n+1-Vogan}. In this case, one calculates that
$$
\Delta^+(\frk, \frt_f)=\{e_i\pm e_j \mid 1\leq i<j\leq n\}\cup\{e_1, \dots, e_n\}
$$
and that
$$
\Delta^+(\frp, \frt_f)=\Delta^+(\frk, \frt_f)\cup \{2e_i\mid 1\leq i\leq n\}.
$$
We have $W(\frg, \frt_f)^1=\{e\}$, $\beta=2e_1$,
$$
\rho_c=(n-1/2, n-3/2, \dots, 1/2), \quad \rho_n=(n+1/2, n-1/2, \dots, 3/2).
$$
Moreover, $\xi_i=\varpi_i=e_1+\cdots+e_i$, $1\leq i\leq n-1$, and $\xi_n=\varpi_n=\frac{1}{2}(e_1+\cdots+e_n)$.

\begin{figure}[H]
\centering
\scalebox{0.6}{\includegraphics{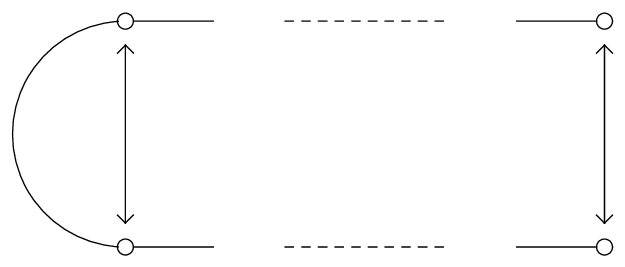}}
\caption{The Vogan diagram for $\frs\frl(2n+1,\bbR)$}
\label{Fig-sl-2n+1-Vogan}
\end{figure}

Let $\mu=[m_1, \dots, m_n]$ be a dominant weight. Then $\mu=(a_1, \dots, a_n)$, where
$$
a_i=\sum_{j=i}^{n-1}m_j + m_n/2, \quad 1\leq i\leq n-1;\quad
a_n=m_n/2.
$$
Note that $a_1\geq \dots \geq a_n$ are simultaneously integers or half-integers.

Guided by Lemma \ref{hyperplane-construction}, we calculate that the $\frk$-type $\mu$ is u-small if and only if
\begin{equation}\label{sl-2n+1R-u-small}
a_1+\cdots+a_k\leq 2nk-k^2+2k, \quad 1\leq k\leq n.
\end{equation}
We \emph{claim} that $\mu$ is u-small when $a_1\leq n+3/2$. Indeed, in this case, we would have that
$$
a_1+\cdots+a_k\leq nk+3k/2\leq 2nk-k^2+2k.
$$
Therefore, the claim follows from \eqref{sl-2n+1R-u-small}.

Now let us assume that $\mu$ is u-large and verify \eqref{pencil-reduction-step}.
As mentioned above, we have $a_1\geq n+2$.   Then
\begin{align*}
\mu-\rho_n &=(a_1-n-1/2,  a_2-n+1/2,  \dots, a_n-3/2),\\
\mu-\beta-\rho_n&=(a_1-n-5/2,  a_2-n+1/2,  \dots, a_n-3/2),
\end{align*}
and
\begin{equation}\label{II-sl-2n+1R}
\|\mu-\rho_n\|^2-\|\mu-\rho_n-\beta\|^2=4(a_1-n-3/2)>0.
\end{equation}
Since $a_1\geq n+2$, we always have $a_1-n-\frac{1}{2}>|a_1-n-\frac{5}{2}|$.
Let $B$ be the multi-set consisting of $|a_2-n+\frac{1}{2}|$, ..., $|a_n-\frac{3}{2}|$. Let $b_1\geq \cdots\geq b_{n}$ be the re-ordering of $a_1-n-\frac{1}{2}$ and $B$. Let $c_1\geq \cdots\geq c_{n}$ be the re-ordering of $|a_1-n-\frac{5}{2}|$ and $B$. Recall that $W(\frk, \frt_f)$ consists of permutations as well as all sign changes. Thus
\begin{align*}
\{\mu-\rho_n\} &=(b_1,   \dots, b_{n-1}, b_n),\\
\{\mu-\beta-\rho_n\}&=(c_1,\dots, c_{n-1}, c_n).
\end{align*}
Since $b_1\geq c_1$, $\dots$, $b_n\geq c_n$ and these inequalities can not happen simultaneously, we have
$$
2\langle \rho_c, \{\mu-\rho_n\}-\{\mu-\beta-\rho_n\} \rangle > 0.
$$
Thus \eqref{pencil-reduction-step} follows from \eqref{II-sl-2n+1R}.

\subsection{$\frs\frl(n,\bbH)$}

This subsection aims to handle $\frs\frl(n,\bbH)$, $n\geq 2$, whose Vogan diagram is presented in Fig.~\ref{Fig-sl-nH-Vogan}. In this case, one calculates that
$$
\Delta^+(\frp, \frt_f)=\{e_i\pm e_j \mid 1\leq i<j\leq n\}
$$
and that
$$
\Delta^+(\frk, \frt_f)=\Delta^+(\frp, \frt_f)\cup \{2e_i\mid 1\leq i\leq n\}.
$$
We have $W(\frg, \frt_f)^1=\{e\}$, $\beta=e_1+e_2$,
$$
\rho_c=(n, n-1, \dots, 2, 1), \quad \rho_n=(n-1, n-2, \dots, 1, 0).
$$
Moreover, $\varpi_i=\xi_i=e_1+\cdots+e_i$, $1\leq i\leq n$.

\begin{figure}[H]
\centering
\scalebox{0.6}{\includegraphics{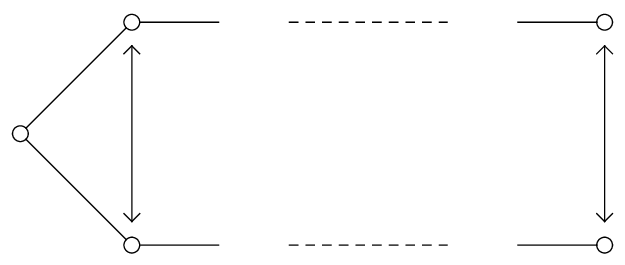}}
\caption{The Vogan diagram for $\frs\frl(n,\bbH)$}
\label{Fig-sl-nH-Vogan}
\end{figure}

Let $\mu=[m_1, \dots, m_n]$ be a dominant weight. Then $\mu=(a_1, \dots, a_n)$, where
$a_i=\sum_{j=i}^{n}m_j$. Guided by Lemma \ref{hyperplane-construction}, we calculate that the $\frk$-type $\mu$ is u-small if and only if
\begin{equation}\label{sl-2n+1R-u-small}
a_1+\cdots+a_k\leq 2nk-k^2-k, \quad 1\leq k\leq n.
\end{equation}
We \emph{claim} that $\mu$ is u-small when $a_1\leq n-1$. Indeed, in this case, we would have that
$$
a_1+\cdots+a_k\leq (n-1)k\leq 2nk-k^2-k.
$$
Therefore, the claim follows from \eqref{sl-2n+1R-u-small}.

Now let us assume that $\mu$ is u-large and verify \eqref{pencil-reduction-step}.
As mentioned above, we have $a_1\geq n$.   Then
\begin{align*}
\mu-\rho_n &=(a_1-n+1,  a_2-n+2,  \dots, a_{n-1}-1, a_n),\\
\mu-\beta-\rho_n&=(a_1-n,  a_2-n+1,  \dots, a_{n-1}-1, a_n),
\end{align*}
and
\begin{equation}\label{II-slnH}
\|\mu-\rho_n\|^2-\|\mu-\rho_n-\beta\|^2=2(a_1+a_2-2n+2).
\end{equation}
Since $W(\frk, \frt_f)$ consists of permutations as well as all sign changes, one sees easily that
$$
2\langle \rho_c, \{\mu-\rho_n\}-\{\mu-\beta-\rho_n\}\rangle>0
$$
when $a_2\geq n-1$. Thus \eqref{pencil-reduction-step} follows from \eqref{II-slnH} whenever $a_2\geq n-1$.

Now assume $a_2\leq n-2$. Then we \emph{claim} that $a_1\geq 2n-1$. Indeed, otherwise,
$$
a_1+\cdots+a_k\leq (2n-2)+(n-2)(k-1)\leq 2nk-k^2-k,
$$
and we would conclude that $\mu$ is u-small. Thus the claim holds, and \eqref{II-slnH} says that
\begin{equation}\label{II-slnH-case2}
\|\mu-\rho_n\|^2-\|\mu-\rho_n-\beta\|^2\geq 2.
\end{equation}
On the other hand, let $B$ be the multi-set of $|a_3-n+3|$, ..., $|a_{n-1}-1|$, $|a_n|$.
Denote the members of $B$ which are greater than $n-2-a_2$ by $b_1\geq \cdots \geq b_k$, and collect the remaining members of $B$ by $c_1\geq \cdots \geq c_t$. Here $t+k=n-2$. Then
\begin{align*}
\{\mu-\rho_n\} &=(a_1-n+1,  b_1,  \dots, b_k,  n-2-a_2, c_1, \dots, c_t),\\
\{\mu-\beta-\rho_n\}&=(a_1-n,  b_1,  \dots, b_k,  n-1-a_2, c_1, \dots, c_t).
\end{align*}
Thus
$$
2\langle \rho_c, \{\mu-\rho_n\}-\{\mu-\beta-\rho_n\}\rangle=2(n-(n-(k+1)))=2(k+1)>0.
$$
Therefore \eqref{pencil-reduction-step} follows from \eqref{II-slnH-case2} whenever $a_2\leq n-2$.

To sum up, Theorem C holds for $\frs\frl(n, \bbH)$, $n\geq 2$.

\section{Exceptional Lie algebras}
This section aims to deal with the exceptional Lie algebras for Theorem C.

\subsection{EI$=E_{6(6)}$}\label{sec-EI}
This subsection aims to handle EI, whose Vogan diagram is presented in Fig.~\ref{Fig-EI-Vogan}. The simple roots for $\Delta^+(\frg, \frt_f)$ are
$$
\alpha_4:=\beta_2, \quad \alpha_3:=\beta_4, \quad \alpha_2:=\frac{1}{2}(\beta_3+\beta_5), \quad \alpha_1:=\frac{1}{2}(\beta_1+\beta_6).
$$
The root system $\Delta^+(\frg, \frt_f)$ is $F_4$, with $\alpha_1, \alpha_2$ short and $\alpha_3, \alpha_4$ long.
On the other hand, $\Delta^+(\frk, \frt_f)$ is $C_4$, and has simple roots
$$
\gamma_1:=\alpha_2+\alpha_3+\alpha_4, \quad \gamma_2:=\alpha_1, \quad \gamma_3:=\alpha_2, \quad \gamma_4:=\alpha_3.
$$
Here $\gamma_4$ is long. One calculates that $W(\frg, \frt_f)^1=\{e, s_{\alpha_4}, s_{\alpha_3+\alpha_4}s_{\alpha_4}\}$ and that
$$
\beta=[0,0,0,1], \quad \rho_n=[5,1,1,0], \quad \rho_n^{(1)}=[3,1,1,1], \quad \rho_n^{(2)}=[1,1,3, 0].
$$

\begin{figure}[H]
\centering
\scalebox{0.5}{\includegraphics{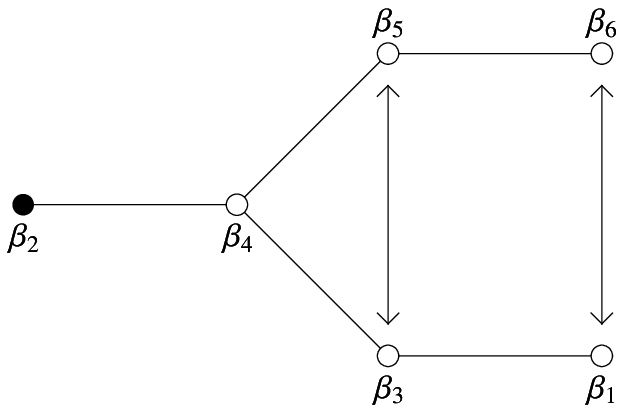}}
\caption{The Vogan diagram for EI}
\label{Fig-EI-Vogan}
\end{figure}

Guided by Lemma \ref{hyperplane-construction}, we calculate that the $\frk$-type $\mu=[a, b, c, d]$ is u-small if and only if
\begin{align*}
&a+2b+2c+2d\leq 18, \quad  2a+3b+4c+4d\leq 34, \quad 3a+4b+5c+6d\leq 48, \\
&a+b+c+d\leq 14, \quad  a+2b+3c+4d\leq24.
\end{align*}
In particular, there are $922$ u-small $\frk$-types in total.

Now let us consider the distribution of the spin norm along pencils. Let $\mu=[a, b, c, d]$ be a u-large $\frk$-type such that $\mu-\beta$ is dominant. Then $a, b, c\geq 0$ and $d\geq 1$. It is easy to calculate that
\begin{equation}\label{II-EI}
\| \mu-\rho_n^{(j)}  \|^2- \| \mu-\beta-\rho_n^{(j)}  \|^2=
\begin{cases}
2(a+2b+3c+4d-12) & \mbox{ if } j=0; \\
2(a+2b+3c+4d-14) & \mbox{ if } j=1, 2.
\end{cases}
\end{equation}

Let us handle the term I defined in \eqref{I}. We \emph{claim} that  it suffices to use elements from $W_1$ to conjugate all the $\mu-\beta-\rho_n^{(j)}$ to $C$ when $a\geq 9$. Let us explain the details for $j=0$.
Note firstly that for any fixed $b, c\geq 0$ and $d\geq 1$, we can find $w\in W_1$
such that
$$
w(\mu-\beta-\rho_n)=\{\mu-\beta-\rho_n\}
$$
when $a$ is big enough. Secondly, let $w_{0, 1}$ be the longest element of $W_1$, then
$$ \langle w_{0, 1}[a-5, -1, -1,
0], \check{\gamma}_1\rangle \leq \langle w[a-5, -1, -1, 0],
\check{\gamma}_1\rangle\leq \langle w[a-5, b-1, c-1, d-1],
\check{\gamma}_1\rangle.
$$
The above first step uses Lemma 7.4  of \cite{D16}, while the second step uses Lemma 7.5 there. But we have
$$ \langle w_{0, 1}[a-5, -1, -1,
0], \check{\gamma}_1\rangle=a-9.
$$
Moreover, when $j=1,2$, we  have
$$
\langle w_{0, 1}[a-3, -1, -1, -1], \check{\gamma}_1\rangle=a-9, \quad
\langle w_{0, 1}[a-1, -1, -3, 0], \check{\gamma}_1\rangle=a-9,
$$
respectively. Thus the claim holds. Similarly,  it suffices to use elements from $W_2$ (resp., $W_3$, $W_4$) to conjugate all the $\mu-\beta-\rho_n^{(j)}$ to $C$ when $b\geq 8$ (resp., $c\geq 7$, $d\geq 8$).

It is direct to check that
$$
2\langle \rho_c, w\beta\rangle \geq -4, \quad \forall w\in W_1,
$$
and that
$$
2\langle \rho_c, w\beta\rangle > 0, \quad \forall w\in W_2, W_3, W_4.
$$
Note that the naive lower bound for I here is $-2\langle \rho_c, \beta\rangle=-20$.
Now in view of \eqref{Diff-j} and \eqref{II-EI}, the inequality \eqref{pencil-reduction-step} holds whenever $a\geq 13$, \emph{or} $b\geq 8$, \emph{or} $c\geq 7$, \emph{or} $d\geq 8$.

Finally, it remains to check \eqref{pencil-reduction-step} for any u-large $\mu=[a, b, c, d]$ such that $0\leq a\leq 12$, $0\leq b\leq 7$, $0\leq c\leq 6$, $1\leq d\leq 7$. This has been carried out on a computer. Thus Theorem C holds for EI.

\subsection{EII$=E_{6(2)}$}\label{sec-EII}
This subsection aims to handle EII, whose Vogan diagram is presented in Fig.~\ref{Fig-EII-Vogan}. In this case, $\Delta^+(\frg, \frt_f)$ is $E_6$, with simple roots
$\alpha_1, \dots, \alpha_6$.  We have that $|W(\frg, \frt_f)^1|=36$.
 We set
 $$
 \gamma_i=\alpha_{7-i}, \, 1\leq i\leq 4; \, \gamma_5=\alpha_1; \, \gamma_6=\alpha_1+2\alpha_2+2\alpha_3+3\alpha_4+2\alpha_5+\alpha_6.
 $$
Then $\beta=[0, 0, 1, 0, 0, 1]$.

\begin{figure}[H]
\centering
\scalebox{0.6}{\includegraphics{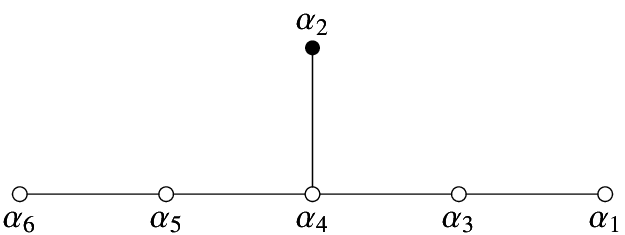}}
\caption{The Vogan diagram for EII}
\label{Fig-EII-Vogan}
\end{figure}

Guided by Lemma \ref{hyperplane-construction}, we calculate that the $\frk$-type $\mu=[a, b, c, d, e, f]$ is u-small if and only if
\begin{align*}
& a+ b+c+ d+e\leq 12, \quad 2a+ 4b+3c+ 2d+e\leq 24,\\
& a+2b+3c+4d+2e\leq 24,   \quad a+2b+3c+4d+5e+3f\leq 60, \\
& a+2b+3c+2d+e+f\leq 24, \quad 5a+10b+9c+8d+7e+3f\leq 96,\\
& 7a+8b+9c+10d+5e+3f\leq 96, \quad  3a+4b+5c+4d+3e+f\leq 44.
\end{align*}
In particular, there are $22122$ u-small $\frk$-types in total.

Now let us consider the distribution of the spin norm along pencils. Let $\mu=[a, b, c, d, e, f]$ be a u-large $\frk$-type such that $\mu-\beta$ is dominant. Then $c, f\geq 1$ and $a, b, d, e\geq 0$. It is easy to calculate that
\begin{equation}\label{II-EII}
\| \mu-\rho_n^{(j)}  \|^2- \| \mu-\beta-\rho_n^{(j)}  \|^2 \geq
a+2b+3c+2d+e+f-14, \quad 0\leq j\leq 35.
\end{equation}

Similar to the EI case, it suffices to use elements from $W_1$ (resp., $W_2$, $W_3$, $W_4$, $W_5$, $W_6$) to conjugate all these $\mu-\beta-\rho_n^{(j)}$ to $C$ when $a\geq 6$ (resp., $b\geq 6$, $c\geq 7$, $d\geq 6$, $e\geq 6$, $f\geq 11$).  Moreover, we have that
$$
2\langle \rho_c, w\beta\rangle \geq -4, \, \forall w\in W_1, W_5; \quad
2\langle \rho_c, w\beta\rangle \geq -8, \, \forall w\in W_6,
$$
and that
$$
2\langle \rho_c, w\beta\rangle > 0, \, \forall w\in W_2, W_3, W_4.
$$
Now in view of \eqref{Diff-j} and \eqref{II-EII}, the inequality \eqref{pencil-reduction-step} holds whenever $a\geq 15$, \emph{or} $b\geq 6$, \emph{or} $c\geq 7$, \emph{or} $d\geq 6$, \emph{or} $e\geq 15$,\emph{ or} $f\geq 20$.

Finally, it remains to check \eqref{pencil-reduction-step} for any u-large $\mu=[a, b, c, d, e, f]$ such that $0\leq a, e\leq 14$, $0\leq b, d\leq 5$, $1\leq c\leq 6$, $1\leq f\leq 19$. This has been carried out on a computer. Thus Theorem C holds for EII.

\subsection{EIV$=E_{6(-26)}$}\label{sec-EIV}
This subsection aims to handle EIV, whose Vogan diagram is presented in Fig.~\ref{Fig-EIV-Vogan}. In this case,
$\Delta^+(\frp, \frt_f)\subset \Delta^+(\frk, \frt_f)$. Thus both $\Delta^+(\frg, \frt_f)$ and $\Delta^+(\frk, \frt_f)$ are $F_4$, with simple roots
$$
\alpha_4:=\beta_2, \quad \alpha_3:=\beta_4, \quad \alpha_2:=\frac{1}{2}(\beta_3+\beta_5), \quad \alpha_1:=\frac{1}{2}(\beta_1+\beta_6).
$$
Here $\alpha_1, \alpha_2$ are short, while $\alpha_3, \alpha_4$ are long. We \emph{identify} $\gamma_i$ with $\alpha_i$, $1\leq i\leq 4$.
We have that $W(\frg, \frt_f)^1=\{e\}$ and that
$$
\beta=[1, 0, 0, 0], \quad \rho_n=[1, 1, 0, 0].
$$

\begin{figure}[H]
\centering
\scalebox{0.5}{\includegraphics{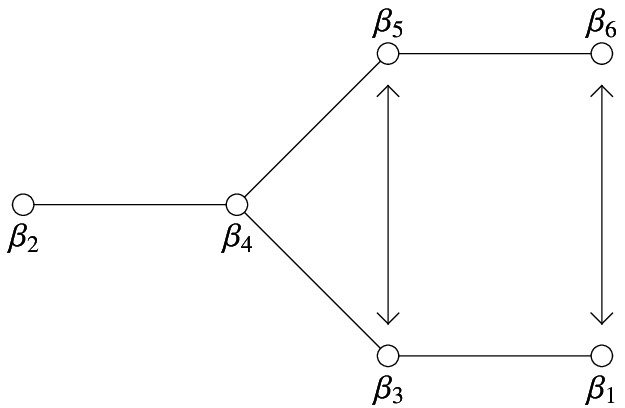}}
\caption{The Vogan diagram for EIV}
\label{Fig-EIV-Vogan}
\end{figure}

Guided by Lemma \ref{hyperplane-construction}, we calculate that the $\frk$-type $\mu=[a, b, c, d]$ is u-small if and only if
\begin{align*}
2a + 3 b + 4 c + 2 d\leq 10, \quad  a + 2 b + 3 c + 2 d\leq 6.
\end{align*}
In particular, there are $37$ u-small $\frk$-types in total. As observed in the author's thesis,  any $\frk$-type whose spin norm is upper bounded by $\|\rho\|$ must be u-small.

Now let us consider the distribution of the spin norm along pencils. Let $\mu=[a, b, c, d]$ be a u-large $\frk$-type such that $\mu-\beta$ is dominant. Then $a\geq 1$ and $b, c, d\geq 0$. Similar to the EI case,   it suffices to use elements from $W_1$ (resp., $W_2$, $W_3$, $W_4$) to conjugate $\mu-\beta-\rho_n$ to $C$ when $a\geq 5$ (resp., $b\geq 2$, $c\geq 2$, $d\geq 3$). Let $w\in W_k$ be an element such that
$$
\{\mu-\beta-\rho_n\}=w(\mu-\beta-\rho_n).
$$
Since $\beta=2\alpha_1+3\alpha_2+2\alpha_3+\alpha_4$, by the technique of Lemma 7.3 of \cite{D16}, we have that $w\beta\in \Delta^+(\frk, \frt_f)$ for any $w\in W_k$, $1\leq k\leq 4$. Moreover,
$$
\{\mu-\rho_n\}-\{\mu-\beta-\rho_n\} =(\{\mu-\rho_n\}-w(\mu-\rho_n)) +w\beta.
$$
Thus \eqref{pencil-reduction-step} holds whenever $a\geq 5$, \emph{or} $b\geq 2$, \emph{or} $c\geq 2$, \emph{or} $d\geq 3$.

Finally, it remains to check \eqref{pencil-reduction-step} for any u-large $\mu=[a, b, c, d]$ such that $1\leq a\leq 4$, $0\leq b\leq 1$, $0\leq c\leq 1$, $0\leq d\leq 2$. This has been carried out on a computer. Thus Theorem C holds for EIV.

\subsection{EV$=E_{7(7)}$}\label{sec-EV}
This subsection aims to handle EV, whose Vogan diagram is presented in Fig.~\ref{Fig-EV-Vogan}. In this case, $\Delta^+(\frg, \frt_f)$ is $E_7$, with simple roots $\alpha_1, \dots, \alpha_7$. Moreover, $\Delta^+(\frk, \frt_f)$ is $A_7$, with simple roots
$$
\gamma_1:=\alpha_1; \quad \gamma_i:=\alpha_{i+1}, \, 2\leq i\leq 6; \quad \gamma_7:=\alpha_1+2\alpha_2+2\alpha_3+3\alpha_4+2\alpha_5+\alpha_6.
$$
We have that $|W(\frg, \frt_f)^1|=72$ and that $\beta=[0,0,0, 1, 0, 0, 0]$.

\begin{figure}[H]
\centering
\scalebox{0.6}{\includegraphics{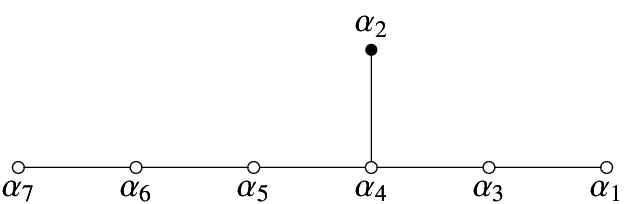}}
\caption{The Vogan diagram for EV}
\label{Fig-EV-Vogan}
\end{figure}

Guided by Lemma \ref{hyperplane-construction}, we calculate that there are $187200$ u-small $\frk$-types in total.

Now let us consider the distribution of the spin norm along pencils. Let
$$
\mu=[a, b, c, d, e, f, g]
$$ be a u-large $\frk$-type such that $\mu-\beta$ is dominant. Then $d\geq 1$ and $a, b, c,  e, f, g\geq 0$. One calculates that
\begin{equation}\label{II-EV}
\| \mu-\rho_n^{(j)}  \|^2- \| \mu-\beta-\rho_n^{(j)}  \|^2 \geq
a+2b+3c+4d+3e+2f+g-20, \, 0\leq j\leq 71.
\end{equation}

The  entry $a\geq 10$ in the following table means that it suffices to use $w\in W_1$ to conjugate all the $\mu-\beta-\rho_n^{(j)}$ to $C$ whenever $a\geq 10$. Other entries of the  first line are interpreted similarly.
\begin{center}
\begin{tabular}{|l|c|c|c|c|c|c|c|}
$W_k$-bound &$a\geq 10$ &   $b \geq 10$ & $c\geq 10$ & $d\geq 11$ & $e\geq 10$ & $f\geq 10$ & $g\geq 10$   \\ \hline
$2\langle\rho_c, w_{0, k}\beta\rangle$&  $-8$ & $0$ & $8$ & $16$ & $8$ & $0$ & $-8$\\
\end{tabular}
\end{center}

Now in view of the above table, \eqref{Diff-j}, \eqref{I-parabolic-lower-bound}, and \eqref{II-EV}, the inequality \eqref{pencil-reduction-step} holds whenever $a\geq 25$, \emph{or} $b\geq 10$, \emph{or} $c\geq 10$, \emph{or} $d\geq 11$, \emph{or} $e\geq 10$, \emph{or} $f\geq 10$, \emph{or} $g\geq 25$.

Finally, it remains to check \eqref{pencil-reduction-step} for any u-large $\mu=[a, b, c, d, e, f, g]$ such that $0\leq a, g\leq 24$, $0\leq b, c, e, f\leq 9$, $1\leq d\leq 10$. This has been carried out on a computer. Thus Theorem C holds for EV.

\subsection{EVI$=E_{7(-5)}$}\label{sec-EVI}
This subsection aims to handle EVI, whose Vogan diagram is presented in Fig.~\ref{Fig-EVI-Vogan}. In this case, $\Delta^+(\frg, \frt_f)$ is $E_7$, with simple roots $\alpha_1, \dots, \alpha_7$. Moreover, $\Delta^+(\frk, \frt_f)$ is $D_6\times A_1$, with simple roots
$$
\gamma_i:=\alpha_{8-i}, \, 1\leq i\leq 4; \, \gamma_5=\alpha_2; \, \gamma_6=\alpha_3;
\, \gamma_7:=2\alpha_1+2\alpha_2+3\alpha_3+4\alpha_4+3\alpha_5+2\alpha_6+\alpha_7.
$$
We have that $|W(\frg, \frt_f)^1|=63$ and that $\beta=[0, 0, 0, 0, 0, 1, 1]$.

\begin{figure}[H]
\centering
\scalebox{0.6}{\includegraphics{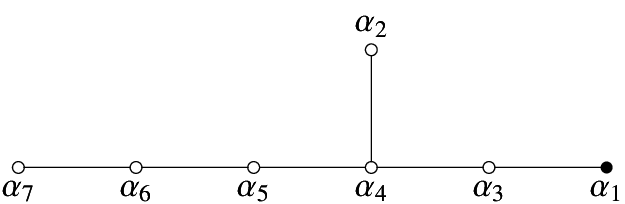}}
\caption{The Vogan diagram for EVI}
\label{Fig-EVI-Vogan}
\end{figure}

Guided by Lemma \ref{hyperplane-construction}, we calculate that there are $105495$ u-small $\frk$-types in total.

Now let us consider the distribution of the spin norm along pencils. Let
$$
\mu=[a, b, c, d, e, f, g]
$$ be a u-large $\frk$-type such that $\mu-\beta$ is dominant. Then $f, g\geq 1$ and $a, b, c,  d, e\geq 0$. One calculates that
\begin{equation}\label{II-EVI}
\| \mu-\rho_n^{(j)}  \|^2- \| \mu-\beta-\rho_n^{(j)}  \|^2 \geq
a+2b+3c+4d+2e+3f+g-20, \, 0\leq j\leq 62.
\end{equation}

The  entry $a\geq 8$ in the following table means that it suffices to use $w\in W_1$ to conjugate all the $\mu-\beta-\rho_n^{(j)}$ to $C$ whenever $a\geq 8$. Other entries of the  first line are interpreted similarly.
\begin{center}
\begin{tabular}{|l|c|c|c|c|c|c|c|}
$W_k$-bound &$a\geq 8$ &   $b \geq 8$ & $c\geq 8$ & $d\geq 8$ & $e\geq 8$ & $f\geq 9$ & $g\geq 17$   \\ \hline
$2\langle\rho_c, w_{0, k}\beta\rangle$&  $-6$ & $2$ & $8$ & $12$ & $4$ & $14$ & $-14$\\
\end{tabular}
\end{center}

Now in view of the above table, \eqref{Diff-j}, \eqref{I-parabolic-lower-bound}, and \eqref{II-EVI}, the inequality \eqref{pencil-reduction-step} holds whenever $a\geq 23$, \emph{or} $b\geq 8$, \emph{or} $c\geq 8$, \emph{or} $d\geq 8$, \emph{or} $e\geq 8$, \emph{or} $f\geq 9$, \emph{or} $g\geq 32$.

Finally, it remains to check \eqref{pencil-reduction-step} for any u-large $\mu=[a, b, c, d, e, f, g]$ such that $0\leq a\leq 22$, $0\leq b, c, d, e\leq 7$, $1\leq f\leq 8$, $1\leq g\leq 31$. This has been carried out on a computer. Thus Theorem C holds for EVI.

\subsection{EVIII$=E_{8(8)}$}\label{sec-EVIII}
This subsection aims to handle EVIII, whose Vogan diagram is presented in Fig.~\ref{Fig-EVIII-Vogan}. In this case, $\Delta^+(\frg, \frt_f)$ is $E_8$, with simple roots $\alpha_1, \dots, \alpha_8$. Moreover, $\Delta^+(\frk, \frt_f)$ is $D_8$, with simple roots
$$
\gamma_1:=2\alpha_1+2\alpha_2+3\alpha_3+4\alpha_4+3\alpha_5+2\alpha_6+\alpha_7;
\,\gamma_i:=\alpha_{10-i}, \, 2\leq i\leq 6; \, \gamma_7=\alpha_2; \, \gamma_8=\alpha_3.
$$
We have that $|W(\frg, \frt_f)^1|=135$ and that $\beta=[0, 0, 0, 0, 0, 0, 1, 0]$.

\begin{figure}[H]
\centering
\scalebox{0.7}{\includegraphics{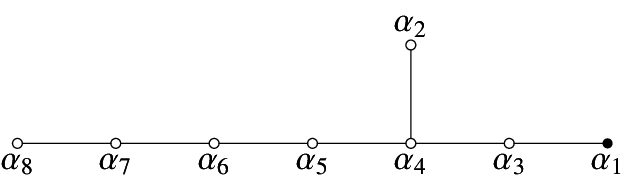}}
\caption{The Vogan diagram for EVIII}
\label{Fig-EVIII-Vogan}
\end{figure}

Guided by Lemma \ref{hyperplane-construction}, we calculate that there are $1379322$ u-small $\frk$-types in total.

Now let us consider the distribution of the spin norm along pencils. Let
$$
\mu=[a, b, c, d, e, f, g, h]
$$ be a u-large $\frk$-type such that $\mu-\beta$ is dominant. Then $g\geq 1$ and $a, b, c, d, e, f, h\geq 0$. One calculates that
\begin{equation}\label{II-EIII}
\| \mu-\rho_n^{(j)}  \|^2- \| \mu-\beta-\rho_n^{(j)}  \|^2 \geq
a+2b+3c+4d+5e+6f+4g+3h-32, \, 0\leq j\leq 134.
\end{equation}

The  entry $a\geq 16$ in the following table means that it suffices to use $w\in W_1$ to conjugate all the $\mu-\beta-\rho_n^{(j)}$ to $C$ whenever $a\geq 16$. Other entries of the  first line are interpreted similarly.
\begin{center}
\begin{tabular}{|l|c|c|c|c|c|c|c|c|}
$W_k$-bound &$a\geq 16$ &   $b \geq 16$ & $c\geq 16$ & $d\geq 16$ & $e\geq 16$ & $f\geq 16$ & $g\geq 17$ & $h\geq 16$  \\ \hline
$2\langle\rho_c, w_{0, k}\beta\rangle$&  $-14$ & $-2$ & $8$ & $16$ & $22$ & $26$ & $28$ & $14$ \\
\end{tabular}
\end{center}

Now in view of the above table, \eqref{Diff-j}, \eqref{I-parabolic-lower-bound}, and \eqref{II-EIII}, the inequality \eqref{pencil-reduction-step} holds whenever $a\geq 43$, \emph{or} $b\geq 16$, \emph{or} $c\geq 16$, \emph{or} $d\geq 16$, \emph{or} $e\geq 16$, \emph{or} $f\geq 16$, \emph{or} $g\geq 17$, \emph{or} $h\geq 16$.

Finally, it remains to check \eqref{pencil-reduction-step} for any u-large $\mu=[a, b, c, d, e, f, g, h]$ such that $0\leq a\leq 42$, $0\leq b, c, d, e, f, h\leq 15$, $1\leq g\leq 16$. This has been done by \texttt{Mathematica} in bout 24 hours. Thus Theorem C holds for EVIII.

\subsection{EIX$=E_{8(-24)}$}\label{sec-EIX}
This subsection aims to handle EIX, whose Vogan diagram is presented in Fig.~\ref{Fig-EIX-Vogan}. In this case, $\Delta^+(\frg, \frt_f)$ is $E_8$, with simple roots $\alpha_1, \dots, \alpha_8$. Moreover, $\Delta^+(\frk, \frt_f)$ is $E_7\times A_1$, with simple roots
$$
\gamma_i:=\alpha_{i}, \, 1\leq i\leq 7; \quad \gamma_8:=2\alpha_1+3\alpha_2+4\alpha_3+6\alpha_4+5\alpha_5+4\alpha_6+3\alpha_7+2\alpha_8.
$$
We have that $|W(\frg, \frt_f)^1|=120$ and that $\beta=[0, 0, 0, 0, 0, 0, 1, 1]$.

\begin{figure}[H]
\centering
\scalebox{0.7}{\includegraphics{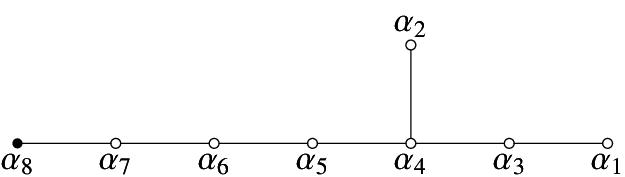}}
\caption{The Vogan diagram for EIX}
\label{Fig-EIX-Vogan}
\end{figure}

Guided by Lemma \ref{hyperplane-construction}, we calculate that there are $577367$ u-small $\frk$-types in total.

Now let us consider the distribution of the spin norm along pencils. Let
$$
\mu=[a, b, c, d, e, f, g, h]
$$ be a u-large $\frk$-type such that $\mu-\beta$ is dominant. Then $g, h\geq 1$ and $a, b, c, d, e, f\geq 0$. One calculates that
\begin{equation}\label{II-EIX}
\| \mu-\rho_n^{(j)}  \|^2- \| \mu-\beta-\rho_n^{(j)}  \|^2 \geq
2a+3b+4c+6d+5e+4f+3g+h-32, \, 0\leq j\leq 119.
\end{equation}

The  entry $a\geq 12$ in the following table means that it suffices to use $w\in W_1$ to conjugate all the $\mu-\beta-\rho_n^{(j)}$ to $C$ whenever $a\geq 12$. Other entries of the  first line are interpreted similarly.
\begin{center}
\begin{tabular}{|l|c|c|c|c|c|c|c|c|}
$W_k$-bound &$a\geq 12$ &   $b \geq 12$ & $c\geq 10$ & $d\geq 10$ & $e\geq 10$ & $f\geq 12$ & $g\geq 13$ & $h\geq 29$  \\ \hline
$2\langle\rho_c, w_{0, k}\beta\rangle$&  $6$ & $14$ & $16$ & $20$ & $22$ & $24$ & $26$ & $-26$ \\
\end{tabular}
\end{center}

Now in view of the above table, \eqref{Diff-j}, \eqref{I-parabolic-lower-bound}, and \eqref{II-EIX}, the inequality \eqref{pencil-reduction-step} holds whenever $a\geq 12$, \emph{or} $b\geq 12$, \emph{or} $c\geq 10$, \emph{or} $d\geq 10$, \emph{or} $e\geq 10$, \emph{or} $f\geq 12$, \emph{or} $g\geq 13$, \emph{or} $h\geq 56$.

Finally, it remains to check \eqref{pencil-reduction-step} for any u-large $\mu=[a, b, c, d, e, f, g, h]$ such that $0\leq a\leq 11$, $0\leq b, f\leq 11$, $0\leq c, d, e\leq 9$, $1\leq g\leq 12$, $1\leq h\leq 55$. This has been done by \texttt{Mathematica} in bout 22 hours. Thus Theorem C holds for EIX.

\subsection{FI$=F_{4(4)}$}\label{sec-FI}
This subsection aims to handle FI, whose Vogan diagram is presented in Fig.~\ref{Fig-FI-Vogan}. In this case, $\Delta^+(\frg, \frt_f)$ is $F_4$, with simple roots
$\alpha_1, \alpha_2, \alpha_3, \alpha_4$. Here $\alpha_1, \alpha_2$ are short, while $\alpha_3, \alpha_4$ are long. Let $s_i$ stand for $s_{\alpha_i}$. We have that
\begin{align*}
W(\frg, \frt_f)^1=\{e, s_4, s_4s_3, s_4s_3s_2, s_4s_3s_2s_1, s_4s_3s_2s_3, s_4s_3s_2s_1s_3, s_4s_3s_2s_3s_4, \\ s_4s_3s_2s_1s_3s_2, s_4s_3s_2s_1s_3s_4, s_4s_3s_2s_1s_3s_2s_3, s_4s_3s_2s_1s_3s_2s_4\}.
\end{align*}
We set
$$
\gamma_i=\alpha_i, \quad 1\leq i\leq 3; \quad \gamma_4=2\alpha_1+4\alpha_2+3\alpha_3+2\alpha_4.
$$
Then $\beta=[0, 0, 1, 1]$, and
\begin{align*}
\rho_n&=[0, 0, 0, 7], \quad \rho_n^{(1)}=[0, 0, 1, 6], \quad \rho_n^{(2)}=[0, 2, 0, 5], \quad \rho_n^{(3)}=[1, 2, 0, 4], \\
\rho_n^{(4)}&=[0, 3, 0, 3], \quad \rho_n^{(5)}=[3, 0, 1, 3], \quad \rho_n^{(6)}=[2, 1, 1, 2], \quad \rho_n^{(7)}=[5, 0, 0, 2],\\
\rho_n^{(8)}&=[2, 0, 2, 1], \quad \rho_n^{(9)}=[4, 1, 0, 1], \quad \rho_n^{(10)}=[0, 0, 3, 0], \quad \rho_n^{(11)}=[4, 0, 1, 0].
\end{align*}

\begin{figure}[H]
\centering
\scalebox{0.6}{\includegraphics{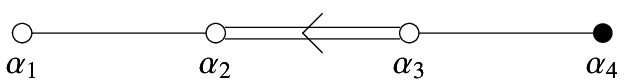}}
\caption{The Vogan diagram for FI}
\label{Fig-FI-Vogan}
\end{figure}

Guided by Lemma \ref{hyperplane-construction}, we calculate that the $\frk$-type $\mu=[a, b, c, d]$ is u-small if and only if
\begin{align*}
& a+b+c\leq 10, \quad a+2 b+2 c\leq 12,  \quad a+b+c+d\leq 14,  \\
&3a+4b+ 5c+d\leq 34, \quad 2 a+3 b+3c+d\leq 24, \quad a+2 b+3c+ d\leq 18.
\end{align*}
In particular, there are $1045$ u-small $\frk$-types in total.

Now let us consider the distribution of the spin norm along pencils. Let $\mu=[a, b, c, d]$ be a u-large $\frk$-type such that $\mu-\beta$ is dominant. Then $c, d\geq 1$ and $a, b\geq 0$. It is easy to calculate that
\begin{equation}\label{II-FI}
\| \mu-\rho_n^{(j)}  \|^2- \| \mu-\beta-\rho_n^{(j)}  \|^2=
\begin{cases}
a+2 b+3 c+ d-9 & \mbox{ if } j=0, 7, 9, 11; \\
a+2 b+3 c+ d-11 & \mbox{ otherwise.}
\end{cases}
\end{equation}

Similar to the EI case, it suffices to use elements from $W_1$ (resp., $W_2$, $W_3$, $W_4$) to conjugate all these $\mu-\beta-\rho_n^{(j)}$ to $C$ when $a\geq 8$ (resp., $b\geq 8$, $c\geq 6$, $d\geq 8$).  Moreover, we have
$$
2\langle \rho_c, w\beta\rangle \geq -1, \, \forall w\in W_1; \quad
2\langle \rho_c, w\beta\rangle \geq -5, \, \forall w\in W_4; \quad
2\langle \rho_c, w\beta\rangle > 0, \, \forall w\in W_2, W_3.
$$
Now in view of \eqref{Diff-j} and \eqref{II-FI}, the inequality \eqref{pencil-reduction-step} holds whenever $a\geq 9$, \emph{or} $b\geq 8$, \emph{or} $c\geq 6$, \emph{or} $d\geq 14$.

Finally, it remains to check \eqref{pencil-reduction-step} for any u-large $\mu=[a, b, c, d]$ such that $0\leq a\leq 8$, $0\leq b\leq 7$, $1\leq c\leq 5$, $1\leq d\leq 13$. This has been carried out on a computer. Thus Theorem C holds for FI.

\subsection{FII$=F_{4(-20)}$}\label{sec-FII}
This subsection aims to handle FII, whose Vogan diagram is presented in Fig.~\ref{Fig-FII-Vogan}. In this case, $\Delta^+(\frg, \frt_f)$ is $F_4$, with simple roots
$\alpha_1, \alpha_2, \alpha_3, \alpha_4$. Here $\alpha_1, \alpha_2$ are short, while $\alpha_3, \alpha_4$ are long. We have that $W(\frg, \frt_f)^1=\{e, s_{\alpha_1}, s_{\alpha_1}s_{\alpha_2}\}$ and that
$$
\beta=[0, 0, 0, 1], \quad \rho_n=[2, 0, 0, 0], \quad \rho_n^{(1)}=[1, 0, 0, 1], \quad \rho_n^{(2)}=[0, 0, 1, 0].
$$

\begin{figure}[H]
\centering
\scalebox{0.6}{\includegraphics{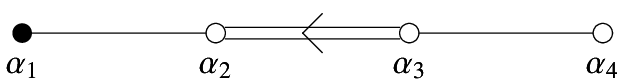}}
\caption{The Vogan diagram for FII}
\label{Fig-FII-Vogan}
\end{figure}

Guided by Lemma \ref{hyperplane-construction}, we calculate that the $\frk$-type $\mu=[a, b, c, d]$ is u-small if and only if
$$
a+2 b+2 c+d\leq 4, \quad  a+2b+ 3 c+2 d\leq 6.
$$
In particular, there are $27$ u-small $\frk$-types in total. As observed in the author's thesis,  any $\frk$-type whose spin norm is upper bounded by $\|\rho\|$ must be u-small.

Now let us consider the distribution of the spin norm along pencils. Let $\mu=[a, b, c, d]$ be a u-large $\frk$-type such that $\mu-\beta$ is dominant. Then $d\geq 1$ and $a, b, c\geq 0$. It is easy to calculate that
\begin{equation}\label{II-FII}
\| \mu-\rho_n^{(j)}  \|^2- \| \mu-\beta-\rho_n^{(j)}  \|^2=
\begin{cases}
a+2 b+3 c+2 d-3 & \mbox{ if } j=0; \\
a+2 b+3 c+2 d-4 & \mbox{ if } j=1, 2.
\end{cases}
\end{equation}
Similar to the EI case, it suffices to use elements from $W_1$ (resp., $W_2$, $W_3$, $W_4$) to conjugate all these $\mu-\beta-\rho_n^{(j)}$ to $C$ when $a\geq 3$ (resp., $b\geq 3$, $c\geq 3$, $d\geq 5$).  Moreover, we have
$$
2\langle \rho_c, w\beta\rangle \geq -1, \quad \forall w\in W_1; \quad
2\langle \rho_c, w\beta\rangle > 0, \quad \forall w\in W_2, W_3, W_4.
$$
Now in view of \eqref{Diff-j} and \eqref{II-FII}, the inequality \eqref{pencil-reduction-step} holds whenever $a\geq 4$, \emph{or} $b\geq 3$, \emph{or} $c\geq 3$, \emph{or} $d\geq 5$.

Finally, it remains to check \eqref{pencil-reduction-step} for any u-large $\mu=[a, b, c, d]$ such that $0\leq a\leq 3$, $0\leq b, c\leq 2$, $1\leq d\leq 4$. This has been carried out on a computer. Thus Theorem C holds for FII.

\subsection{$G_{2(2)}$}

This subsection aims to handle $G_{2(2)}$, whose Vogan diagram is presented in Fig.~\ref{Fig-G-Vogan}, where $\alpha_1=(1,-1,0)$ is short, while $\alpha_2=(-2, 1, 1)$ is long. In this case, $\Delta^+(\frg, \frt_f)$ is $G_2$, while $\Delta^+(\frk, \frt_f)$ is $A_1\times A_1$. Indeed,
$\Delta^+(\frk, \frt_f)$ consists of two orthogonal roots: $\gamma_1:=\alpha_1$, $\gamma_2:=3\alpha_1+2\alpha_2$. One calculates that
$$
\xi_1=2\alpha_1+\alpha_2, \quad \xi_2=3\alpha_1+2\alpha_2,
$$
and that
$$
\varpi_1=(1/2, -1/2, 0),\quad \varpi_2=(-1/2, -1/2, 1).
$$
Moreover, we have that $W(\frg, \frt_f)^1=\{e, s_{\alpha_2}, s_{\alpha_1+\alpha_2}s_{\alpha_2}\}$, and that
$$
\beta=[3, 1], \quad \rho_n=[0,2], \quad \rho_n^{(1)}=[3, 1], \quad \rho_n^{(2)}=[4, 0].
$$

\begin{figure}[H]
\centering
\scalebox{0.5}{\includegraphics{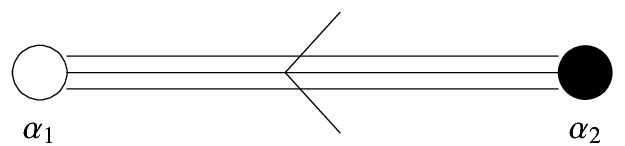}}
\caption{The Vogan diagram for $G_{2(2)}$}
\label{Fig-G-Vogan}
\end{figure}

Guided by Lemma \ref{hyperplane-construction}, we calculate that the $\frk$-type $\mu=[a, b]$ is u-small if and only if
$$
a+3b\leq 12, \quad a+b\leq 8.
$$
Thus we can draw the picture of the u-small convex hull as in Fig.~\ref{Fig-G-USmall}.
In particular, there are $29$ u-small $\frk$-types in total. Note that any $\frk$-type whose spin norm is upper bounded by $\|\rho\|$ must be u-small.

Now let us consider the distribution of the spin norm along pencils. It is easy to calculate that
\begin{equation}\label{II-G}
\| \mu-\rho_n^{(j)}  \|^2- \| \mu-\beta-\rho_n^{(j)}  \|^2=
\begin{cases}
3(a+b-4) & \mbox{ if } j=0; \\
3(a+b-6) & \mbox{ if } j=1, 2.
\end{cases}
\end{equation}
Assume that $\mu=[a, b]$ is u-large and that $\mu-\beta=[a-3, b-1]$ is dominant. In particular, $a\geq 3$, $b\geq 1$. Note that
$$
\mu-\beta-\rho_n=[a-3, b-3], \quad \mu-\beta-\rho_n^{(1)}=[a-6, b-2], \quad \mu-\beta-\rho_n^{(2)}=[a-7, b-1].
$$
The above expressions suggest that when $a\geq 7$ (resp. $b\geq 3$), we need only to use elements in the parabolic subgroup $W_1=\{e, s_{\gamma_2}\}$ (resp., $W_2=\{e, s_{\gamma_1}\}$) of $W(\frk, \frt_f)$ to conjugate all these $\mu-\beta-\rho_n^{(j)}$ to $C$. It is direct to check that
$$
2\langle \rho_c, w\beta\rangle\geq 0, \quad \forall w\in W_1, W_2.
$$
Note that the naive lower bound for I here is $-2\langle \rho_c, \beta\rangle=-6$.
Now in view of \eqref{Diff-j} and \eqref{II-G}, the inequality \eqref{pencil-reduction-step} holds whenever $a\geq 7$, \emph{or} $b\geq 4$.

Finally, the inequality \eqref{pencil-reduction-step} has been checked for any u-large $\mu=[a, b]$ such that $3\leq a\leq 6$ \emph{and} $1\leq b\leq 3$. Thus Theorem C holds for $G_{2(2)}$.

\section{${\rm Sp}(4, \bbR)$}\label{sec-Sp4R}
This section aims to consider $G(\bbR)={\rm Sp}(4, \bbR)$.
Then $K(\bbR)={\rm U}(2)$ and $T(\bbR)_f={\rm U}(1)\times {\rm U}(1)$. Thus $\frk$ \textbf{has center}.
We fix
$$
\Delta^+(\frk, \frt_f)=\{(1, -1)\}, \quad \Delta^+(\frp, \frt_f)=\{(2, 0), (0, 2), (1, 1)\}.
$$
Note that $\beta_1=(2, 0)$ and $\beta_2=(0, -2)$ are the two highest weights of the $K(\bbR)$-representation of $\frp$. Moreover, one calculates that
$$
\rho_n=(3/2, 3/2), \quad \rho_n^{(1)}=(3/2, -1/2),
\quad \rho_n^{(2)}=(1/2, -3/2),
\quad \rho_n^{(3)}=(-3/2, -3/2).
$$
Let $(p, q)$ denote the highest weight of a $K(\bbR)$-type, where $p\geq q$ are two integers. Then by Example 6.3 of \cite{SV}, we have that the $K(\bbR)$-type $(p, q)$ is u-small if and only if
$$
p\leq 3, \quad -3\leq q, \quad  p-q\leq 4.
$$
Thus we can draw $R(\Delta(\frp, \frt_f))\cap C$ in Fig.~\ref{Fig-Sp4R-USmall}, where the dotted circles stand for u-small $K(\bbR)$-types. Note that $2\rho_n$, $2\rho_n^{(1)}$, $2\rho_n^{(2)}$, $2\rho_n^{(3)}$ there are extremal points of the u-small convex hull. There are $25$ u-small $K(\bbR)$-types in total.

\begin{figure}[H]
\centering
\scalebox{0.76}{\includegraphics{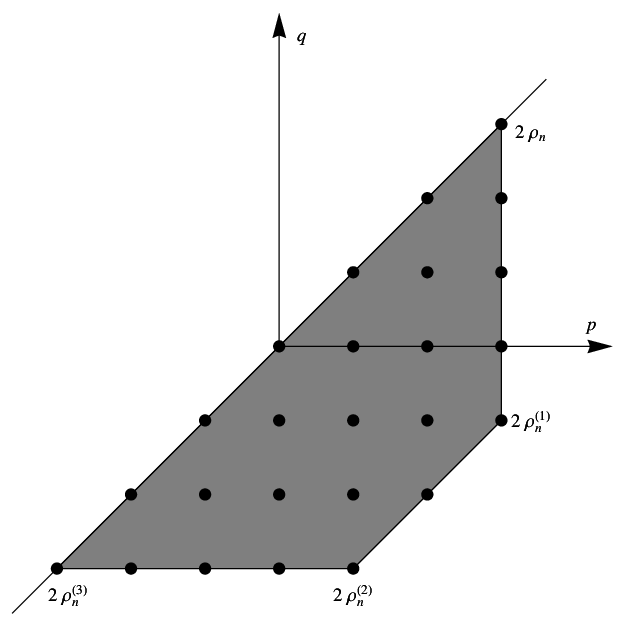}}
\caption{The ${\rm Sp}(4, \bbR)$ case}
\label{Fig-Sp4R-USmall}
\end{figure}

Take $\mu_m=(-m+2, -m)$, where $m\in\bbZ_{>0}$. Assume that $m$ is big enough. Then $\mu_m$ is u-large. Moreover, $\mu_m-\beta_1=(-m, -m)$ and $\mu_m-\beta_2=(-m+2, -m+2)$ are both dominant. One calculates that
\begin{align*}
\|\mu_m-\beta_2\|_{\rm spin}=\sqrt{2m^2-14m+25}&<\|\mu_m\|_{\rm spin}=\sqrt{2m^2-10m+17}\\
&<\|\mu_m-\beta_1\|_{\rm spin}=\sqrt{2m^2-6m+5}.
\end{align*}

Similarly, take $\mu_m^\prime=(m+2, m)$, where $m\in\bbZ_{>0}$. Assume that $m$ is big enough. Then $\mu_m^\prime$ is u-large. Moreover, $\mu_m^\prime-\beta_1=(m, m)$ and $\mu_m^\prime-\beta_2=(m+2, m+2)$ are both dominant. One calculates that
\begin{align*}
\|\mu_m^\prime-\beta_1\|_{\rm spin}=\sqrt{2m^2-6m+5}&<\|\mu_m^\prime\|_{\rm spin}=\sqrt{2m^2+2}\\
&<\|\mu_m^\prime-\beta_2\|_{\rm spin}=\sqrt{2m^2+2m+1}.
\end{align*}

\section{Computing some scattered members of $\widehat{G(\bbR)}^d$}\label{sec-algorithm}
This section aims to develop an explicit algorithm for computing some scattered members of $\widehat{G(\bbR)}^d$.
To be more precise, we will pin down every $\pi\in\widehat{G(\bbR)}^d$  such that it is \emph{not} cohomologically induced from a \emph{weakly good} module on a \emph{proper} $\theta$-stable Levi subgroup. Let $p=(x, \lambda, \nu)$ be the \texttt{atlas} parameter of  such a $\pi$. Based on Section 3, we proceed as follows:
\begin{itemize}
\item[(a)] Enumerate all the \emph{dominant} real infinitesimal characters $\Lambda$ satisfying \eqref{Lambda-bound} and that are conjugate to $\delta+\rho_c$ for certain $\frk$-type $\delta$.
\item[(b)] Enumerate all the KGB elements of $G(\bbR)$ that are fully supported.
\item[(c)] For each KGB element $x$ in (b), pick up from (a) those $\Lambda$ such that \eqref{nu-bound} holds for $\nu=\frac{\Lambda-\theta_x(\Lambda)}{2}$.
\item[(d)] For each KGB element $x$  in step (b), and for each $\Lambda$  in step (c), enumerate all the irreducible representations of $G(\bbR)$ with KGB element $x$ and infinitesimal character $\Lambda$ via the command
\begin{verbatim}
set all=all_parameters_x_gamma(x, Lambda)
\end{verbatim}
Further select the unitary ones out of the  above modules  via the command
\begin{verbatim}
for p in all do if is_unitary(p) then prints(p) fi od
\end{verbatim}
\item[(e)] For the modules surviving in step (d), check  whether they have Dirac cohomology or not via Theorem \ref{thm-HP}. More precisely, given the infinitesimal character $\Lambda$, this theorem allows us to enumerate all the (finitely many) $K(\bbR)$-types \texttt{CanK} that can possibly contribute to Dirac cohomology. Calculate the maximum \texttt{atlas} height $\texttt{ht}$ of the $K(\bbR)$-types in \texttt{CanK}. Then look at the $K(\bbR)$-types of the concerned representation up to this height via the command
\begin{verbatim}
branch_irr(p, ht)
\end{verbatim}
The irreducible unitary representation $\pi$ has non-zero Diac cohomology if and only if at least one $K(\bbR)$-type in \texttt{CanK} shows up in the output of the above command.
\end{itemize}

Step (b) above uses Theorem \ref{thm-Vogan}. Indeed, if $x$ is not fully supported, then any irreducible representation $\pi$ with KGB element $x$ must be cohomologically induced from a weakly good module on a proper $\theta$-stable Levi subgroup. Thus $\pi$ is not among the ones that we are seeking for.

Another remark is that if the group $G(\bbR)$---which must be simple by our assumptions---has trivial center, then any one-dimensional unitary character of $G(\bbR)$ must be trivial. Therefore, in such a case, to find the non-trivial scattered members of $\widehat{G(\bbR)}^d$,  Theorem \ref{thm-SR} allows us to focus on those $\Lambda$ which are \emph{not} strongly regular in step (a). This will reduce the workload significantly.

Let us illustrate this algorithm for the linear ${\rm EIV}=E_{6(-26)}$, which is realized
in \texttt{atlas}  as \texttt{E6\_F4}.
This group is centerless, connected and simply connected. Up to conjugation, this group has a unique $\theta$-stable Cartan subgroup. We adopt the root systems as in Section \ref{sec-EIV}. The algorithm now runs as follows:
\begin{itemize}
\item[$\bullet$]
Step (a) gives us $1147419$ candidates for $\Lambda$, among which $105003$ are not strongly regular. By Theorem \ref{thm-SR}, it suffices to focus on the latter ones.

\item[$\bullet$] EIV has $45$ KGB elements in total (two of which are listed below). The following ones are fully supported
$$
\#x=12, 14, 16, 17, 19, 20, 21, \quad  23\leq \#x\leq 44.
$$

\item[$\bullet$] Now say fix $\#x=19$. Then only $203$ infinitesimal characters from the first step meet the criterion \eqref{nu-bound}.
Only one representation survives after carrying out steps (d) and (e). This gives the first row of Table \ref{table-EIV-scattered-part}. This representation has infinitesimal character $[\frac{1}{2}, 1, \frac{1}{2}, 1, \frac{1}{2}, \frac{1}{2}]$. Note that here $\texttt{atlas}$ uses the fundamental weights of $\Delta^+(\frg, \frh_f)$ as a basis to express $\lambda$, $\nu$ and the infinitesimal character. This representation has a unique lambda lowest $K(\bbR)$-type $[2,0,0,0]$, which differs from its unique spin lowest $K(\bbR)$-type $[1,1,0,0]$.
\item[$\bullet$] All other fully supported KGB elements produce no non-trivial scattered members of $\widehat{\rm EIV}^d$.
\end{itemize}
To sum up, the scattered members of $\widehat{\rm EIV}^d$ that we have obtained are given in Table \ref{table-EIV-scattered-part}, where in the second row sits the trivial representation.

\begin{table}[H]
\centering
\caption{Some scattered members of $\widehat{\rm EIV}^{\mathrm{d}}$}
\begin{tabular}{l|c|c|c|c|r}
$\# x$ &   $\lambda$  & $\nu$ & spin LKTs & mult & u-small\\
\hline
$19$ & $[1,2,0,1,0,1]$ & $[\frac{3}{2},3,-\frac{3}{2},0,-\frac{3}{2},\frac{3}{2}]$ & $[1, 1, 0, 0]$ & $1$ & Yes\\
$44$ & $[1,1,1,1,1,1]$ & $[4,0,0,0,0,4]$ & $[0, 0, 0, 0]$ & $1$ & Yes
\end{tabular}
\label{table-EIV-scattered-part}
\end{table}

Some information of the two KGB elements involved in Table \ref{table-EIV-scattered-part} is listed below. Example 14.19 of \cite{AC}  carefully explains the entries.
\begin{verbatim}
19:   5  [C,C,C,c,C,C]   12  13  23  19  24  14    1x2x4x3x1xe
44:  12  [C,c,c,c,c,C]   42  44  44  44  44  43    1x3x4x2x6x5x4x2x3x1x4x3xe
\end{verbatim}

In a subsequent paper, we will show that with additional effort one can completely pin down $\widehat{\rm EIV}^d$, and that Table \ref{table-EIV-scattered-part} turns out to exhaust \emph{all} the scattered members of $\widehat{\rm EIV}^d$. Actually, we plan to report $\widehat{G(\bbR)}^d$ for several real exceptional Lie groups in future. In particular, the size of the scattered part of $\widehat{G(\bbR)}^d$ turns out to be closely related to the number of u-small $K(\bbR)$-types. For instance, recall from Section \ref{sec-EIV} that \texttt{E6\_F4} has only $37$ u-small $K(\bbR)$-types. On the other hand, the algorithm will give us twenty two scattered members of $\widehat{\rm FI}^d$. Here by FI we actually mean the group \texttt{F4\_s} in $\texttt{atlas}$. This group is centerless, connected, but \emph{not} simply connected. It has $544$ u-small $K(\bbR)$-types (recall from Section \ref{sec-FI} that the number for its universal covering group is $1045$).

\end{document}